\documentclass[final,12pt]{colt2025}

\usepackage[letterpaper,margin=1in]{geometry}  % Letter size paper and 1-inch margins
\usepackage{setspace}                         % For controlling line spacing
\usepackage{amssymb}         % For math environments
\usepackage{titlesec}                         % For customizing section titles
\usepackage{graphicx}                         % For including graphics if needed
\usepackage{hyperref}                         % For hyperlinks in the document
\usepackage{enumitem}                         % For customizable lists
\usepackage{booktabs}                         % For better table formatting
\usepackage{wrapfig}
\usepackage[bf]{caption}

\titlespacing*{\section}{0pt}{\baselineskip}{0.5\baselineskip}
\titlespacing*{\subsection}{0pt}{\baselineskip}{0.5\baselineskip}
\usepackage[english]{babel}
\usepackage{xcolor}
\usepackage{amsmath, bm, mathtools, mathdots}
\usepackage{bbm}
\usepackage{newtxmath}
\usepackage{tikz-cd}
\usepackage{graphicx} % takes care of graphic including machinery
\usepackage[all]{xy}
\usepackage{tikz}
\usepackage{natbib}
\usepackage[notransparent]{svg}
\newcommand{\R}{\mathbb{R}} %\R for reelle tal
\newcommand{\N}{\mathcal{N}} %\N for naturlige tal
\newtheorem{assumption}{Assumption}
\usepackage{wasysym}

\newtheorem{my_definition}{Definition}
\newtheorem{my_corollary}{Corollary}
\newtheorem{my_proposition}{Proposition}
\newtheorem{my_lemma}{Lemma}
\newtheorem{my_remark}{Remark}

\usepackage{cleveref}

\crefname{assumption}{assumption}{assumptions}
\Crefname{assumption}{Assumption}{Assumptions}
\crefname{my_corollary}{corollary}{corollaries}
\Crefname{my_corollary}{Corollary}{Corollaries}
\crefname{my_lemma}{lemma}{lemmas}
\Crefname{my_lemma}{Lemma}{Lemmas}
\crefname{my_remark}{remark}{remarks}
\Crefname{my_remark}{Remark}{Remarks}
\crefname{my_definition}{definition}{definitions}
\Crefname{my_definition}{Definition}{Definitions}
\crefname{my_proposition}{proposition}{propositions}
\Crefname{my_proposition}{Proposition}{Propositions}

\newcommand{\Poincare}{\text{Poincar\'e}}
\newcommand{\GibbsMeasure}{{\mu_\epsilon}}
\newcommand{\TruncatedGibbsMeasure}{{\mu_{\epsilon, U}}}
\newcommand{\LebesgueMeasure}[1]{{\mu_{#1}}}
\newcommand{\VAR}{\mathrm{Var}}
\newcommand{\Dirichletform}{\mathcal{D}}
\newcommand{\sobolevspace}[1]{\mathrm{H}^{1}}
\newcommand{\ud}{\mathrm{d}}
\newcommand{\optimalset}[1]{{{S}^{#1}}}
\newcommand{\optimalsetinterior}[1]{S^{#1}}
\newcommand{\dist}{\mathrm{dist}}
\newcommand{\udiv}{\mathrm{div}}
\let\det\relax
\newcommand{\det}{\mathrm{det}}
\newcommand{\PLconstant}{\nu}
\newcommand{\NeummanEigenvalue}{\lambda_1^{n}}
\newcommand{\Eigenvalue}{\lambda_1}
\newcommand{\PIconstant}{\rho}

\newcommand{\errorboundconstant}{\nu_{eb}}

\newcommand{\generator}{\mathcal{L}}
\newcommand{\red}[1]{{#1}}
\newcommand{\M}{\mathcal{M}} %\M for rationale tal
\newcommand{\logPLmeasure}{\texttt{Log-PL$^\circ$}}
\newcommand{\PLcirc}{\texttt{PL$^\circ$}}

\newcommand{\optional}[1]{}

\title[Poincar\'e Inequality for Local Log-Polyak-{\L}ojasiewicz Measures]{Poincar\'e Inequality for Local Log-Polyak-{\L}ojasiewicz Measures
: Non-asymptotic Analysis in Low-temperature Regime
}
\coltauthor{%
 \Name{Yun Gong} \Email{student\textunderscore gongyun@pku.edu.cn}\\
 \addr Peking University
 \AND
 \Name{Niao He} \Email{niao.he@inf.ethz.ch}\\
 \addr ETH Zürich%
 \AND
 \Name{Zebang Shen} \Email{zebang.shen@inf.ethz.ch}\\
 \addr ETH Zürich%
}
\date{}
\begin{document}
\maketitle
\begin{abstract}

\noindent Potential functions in highly pertinent applications, such as deep learning in over-parameterized regime, are empirically observed to admit non-isolated minima.
To understand the convergence behavior of stochastic dynamics in such landscapes, we propose to study the class of \logPLmeasure\ measures $\mu_\epsilon \propto \exp(-V/\epsilon)$, where the potential $V$ satisfies a local Polyak-Łojasiewicz (P\L) inequality, and its set of local minima is provably \emph{connected}. 
% We establish the Poincaré inequality (PI) for a class of Gibbs measures $\mu_\epsilon \propto \exp(-V/\epsilon)$, where the potential $V$ satisfies a local Polyak-Łojasiewicz (P\L) inequality, and its set of local minima is \emph{connected}. 
Notably, potentials in this class can exhibit local maxima and we characterize its optimal set $\optimalset{}$ to be a compact $\mathcal{C}^2$ \emph{embedding submanifold} of $\mathbb{R}^d$ without boundary.
The \emph{non-contractibility} of $\optimalset{}$ distinguishes our function class from the classical convex setting topologically.
Moreover, the embedding structure induces a naturally defined Laplacian-Beltrami operator on $\optimalset{}$, and we show that its first non-trivial eigenvalue provides an \emph{$\epsilon$-independent} lower bound for the \Poincare\ constant in the \Poincare\ inequality of $\mu_\epsilon$.
As a direct consequence, Langevin dynamics with such non-convex potential $V$ and diffusion coefficient $\epsilon$ converges to its equilibrium $\mu_\epsilon$ at a rate of $\tilde{\mathcal{O}}(1/\epsilon)$, provided $\epsilon$ is sufficiently small. Here $\tilde{\mathcal{O}}$ hides logarithmic terms.

% Our results hold for sufficiently small temperature parameters $\epsilon$. Notably, the potential $V$ can exhibit local maxima, and its optimal set may be \emph{non-contractible}, distinguishing our function class from the convex setting.

% We consider two scenarios for the optimal set $\optimalset{}$: (1) $\optimalset{}$ has interior in $\mathbb{R}^d$ with a Lipschitz boundary, and (2) $\optimalset{}$ is a compact $\mathcal{C}^2$ embedding submanifold of $\mathbb{R}^d$ without boundary. In these cases, the Poincaré constant is bounded below by the spectral properties of differential operators on $\optimalset{}$—specifically, the smallest Neumann eigenvalue of the Laplacian in the first case and the smallest eigenvalue of the Laplace-Beltrami operator in the second. These quantities are temperature-independent.

% Our proof leverages the Lyapunov function approach introduced by \citet{PI_Lyapunov}, reducing the verification of the PI to the stability of the spectral gap of the Laplacian (or Laplacian-Beltrami) operator on $\optimalset{}$ under domain expansion. We establish this stability through carefully designed expansion schemes, which is key to our results.

% \emph{Empahsize the take-away message: Langevin on \logPLmeasure is the same as (Riemannian) Brownian motion on $S$.}
\end{abstract}
\begin{keywords}%
  \Poincare\ inequality, non-log-concave measure%
\end{keywords}
\section{Introduction}
Consider the Langevin dynamics
\begin{equation} \label{eqn_langevin_dynamics}
    \setlength{\abovedisplayskip}{6pt}
    \setlength{\belowdisplayskip}{6pt}
    \ud X(t) = -\nabla V(X(t))\ud t + \sqrt{2\epsilon} \ud W(t), 
\end{equation}
where $V\in \mathcal{C}^2(\R^d,\R)$ and $\epsilon > 0$ are the \emph{potential} and \emph{temperature} of the above system respectively, and $W(t)$ denotes the $d$-dimensional Brownian motion.
% Without loss of generality, we assume throughout the paper $V^* := \min_{x\in\R^d} V(x) = 0$.
Under mild conditions, the above Stochastic Differential Equation (SDE) yields a unique equilibrium $\GibbsMeasure$, commonly known as the Gibbs measure:
\begin{equation} \label{eqn_gibbs}
    \setlength{\abovedisplayskip}{5pt}
    \setlength{\belowdisplayskip}{5pt}
    \GibbsMeasure(x) = \frac{\exp(-V(x)/\epsilon)}{Z_{\epsilon}},\text{ where } Z_\epsilon = \int_{\R^d} \exp(-V(x)/\epsilon)\ud x.
\end{equation}
Langevin dynamics have various applications in domains like statistics, optimization, and machine learning, including sampling from target distribution $\GibbsMeasure$ \citep{wibisono2018sampling}, minimizing non-convex objectives \citep{raginsky2017non,zhang2017hitting},
% estimating the normalizing factor $Z_\epsilon$ in \cref{eqn_gibbs} \red{ref!}, 
% analyzing the differential private algorithms \citep{chourasia2021differential}, 
modelling Stochastic Gradient Descent (SGD) through SDE approximation \citep{li2017stochastic,ben2022high,paquette2022homogenization,li2024hessian}. 
% In this work, we focus on the low temperature regime, i.e. $\epsilon$ is close to zero, a challenging but practically relevant scenario. For example, in SGD approximations, $\epsilon$ corresponds to the step size of SGD which is typically small; when sampling from a posterior distribution, a low temperature allows the generated samples concentrate around the modes.
In practice, the low-temperature regime ($\epsilon\!\! \ll\!\! 1$) is particularly relevant: In the context of sampling, low temperatures enable sharper concentration of samples around the modes, while in the context of SGD approximations, $\epsilon$ corresponds to the step size, which is typically small.

An important aspect of Langevin dynamics is its convergence behavior toward equilibrium, also known as ergodicity \citep{Cattiaux2016HittingTF}.
This is often studied through functional inequalities like the \Poincare\ inequality (PI) and the Log-Sobolev inequality (LSI), which {quantify} the rate of convergence.
In this work, we focus on PI. Here, we recall that a measure $\mu$ satisfies PI with \Poincare\ constant $\PIconstant_\mu$ (formally defined in \Cref{definition_PI}) if for any test function $f$ in the Sobolev space weighted by $\mu$, its variance times $\PIconstant_\mu$ is bounded by its Dirichlet energy (both measured w.r.t. $\mu$).
% defined as follows.
% \begin{my_definition}[\Poincare-Wirtinger Inequality] \label{definition_PI}
%     A probability measure $\mu$ with support $\Omega\subseteq\R^d$ satisfies the Poincaré inequality with parameter $\PIconstant_\mu$, or shortly PI$(\rho_\mu)$, if one has for any $f\in \sobolevspace{1,2}(\mu)$ 
%     \begin{equation} \label{eqn_poincare_inequality}
%         \left\{\VAR_\mu(f) := \int_{\Omega} \left(f - \int_{\Omega} f \ud \mu\right)^2 \ud \mu\right\} \leq \frac{1}{\PIconstant_\mu} \left\{\Dirichletform_\mu(f) := \int_{\Omega} |\nabla f|^2 \ud \mu\right\},
%     \end{equation}
%     where $\PIconstant_\mu$ is called the \emph{PI constant} and $\sobolevspace{1,2}(\mu)$ denotes the Sobolev space weighted by $\mu$.
% \end{my_definition}
% \noindent Here $\VAR_\mu(f)$ and $\Dirichletform_\mu(f)$ are the variance and the Dirichlet energy of the test function $f\in\sobolevspace{1,2}(\mu)$ w.r.t. $\mu$. 
Clearly, the PI constant of $\GibbsMeasure$ is a function of the temperature $\epsilon$. 

The convergence of Langevin dynamics under the \(\chi^2\)-divergence is closely tied to the constant \(\PIconstant_{\GibbsMeasure}\). Specifically, to achieve \(\chi^2(X(t), \GibbsMeasure) \leq \omega\), the required time is \(t = \mathcal{O}\left(\frac{1}{\epsilon \PIconstant_{\GibbsMeasure}} \log \frac{1}{\omega}\right)\).
Hence, in the low-temperature region, the dependence on $\epsilon$ is the determining factor in the rate of convergence, which will be the \emph{main focus of the paper}.
Notably, this convergence behavior varies significantly between uni-modal and multi-modal distributions, and is reflected in the dependence of \(\PIconstant_{\GibbsMeasure}\) on \(\epsilon\):
\vspace{-.1cm}
\begin{itemize}[leftmargin=*]
    \setlength\itemsep{0em}
    \item (constant-time convergence) When $V$ is a strongly convex function or is close to one up to a perturbation of order $\epsilon$, a classical result is that $\PIconstant_\GibbsMeasure$ is of order $\Omega(\frac{1}{\epsilon})$ \citep{Bakry2013AnalysisAG}.
    Hence the mixing time $t = \tilde{\mathcal{O}}(1)$ for all low temperatures.
    \item (sub-exponential-time convergence) For a general log-concave measure, existing research primarily focuses on how the PI constant depends on the problem dimension $d$, commonly known as the Kannan-Lovász-Simonovits (KLS) conjecture \citep{chen2021almost,lee2024eldan}, but gives little emphasis to its dependence on $\epsilon$. Nevertheless, existing technique can be combined to show that $\PIconstant_\GibbsMeasure = \Omega(1)$ under a mild exponential integrability assumption, i.e. $\int_{\R^d} \exp(-V(x))\ud x < \infty$. See a proof in \Cref{section_pi_log_concave}. Moreover, one can easily construct a convex function such that the corresponding $\PIconstant_\GibbsMeasure$ is constant.
    Consequently, for the general class of log-concave measures, $\PIconstant_\GibbsMeasure = \Theta(1)$, and hence, the mixing time is $t = \tilde{\mathcal{O}}(\frac{1}{\epsilon})$, i.e. sub-exponential.
    
    \item (exponential-time convergence) When $V$ has at least two \emph{separated} local minima, convergence occurs in two distinct time-scales: a sub-exponential-time scale describes $X(t)$ reaching a meta-stable equilibrium in one of $V$'s local regions of attraction; and an exponential scale describes the transition between meta-stable equilibria, which typically takes $\Omega(\exp(\frac{1}{\epsilon}))$ time \citep{bovier2004metastability,gayrard2005metastability,AOP}. This exponential-time estimation is commonly known as the the Eyring-Kramers law \citep{eyring1935activated,kramers1940brownian}.
    % If $V$ is a general non-convex function, under mild growth conditions outside a compact set, one can estimate $\PIconstant_\GibbsMeasure = \Theta(\exp(\frac{1}{\epsilon}))$, by directly utilizing the Holley-Stroock perturbation principle on the said compact set. 
    % More refined analysis can be made for Morse functions, but the exponential dependence remains the same \citep{AOP}.
    % In particular, this estimation is tight by the Eyring-Kramers law when $\GibbsMeasure$ is multi-modal, i.e. $V$ has at least two separated (local) minima \citep{eyring1935activated,kramers1940brownian}.
\end{itemize}
\vspace{-.1cm}
{A more detailed literature review is deferred to \Cref{section_related_work}.}

In this paper, we focus on uni-modal measures, which allows convergence to occur within sub-exponential time\footnote{Our analysis can also address convergence toward a metastable equilibrium for multi-modal measures if the domain is properly partitioned, and appropriate boundary conditions are imposed. However, this is orthogonal to the primary focus of this project and is therefore not discussed here.}.
Multi-modal measures, as discussed above, exhibit global convergence on an exponential time scale, a behavior which is less relevant in practical applications. We leave its investigation for future work.

For uni-modal measures, existing research often assumes that the optimal set have a simplistic structure, such as a singleton (when \( V \)  is strongly convex) or a convex set (when \( V \) is convex). These assumptions, while analytically convenient, limit the scope of applications.
% For uni-modal measures, we note that for the potential functions typically studied in the literature, their optimal sets exhibit very simple topological or geometrical structures, being either singletons (when \( V \) is strongly convex) or convex sets (when \( V \) is convex). 
For example, in key applications like deep learning, the function \( V \) is high-dimensional, highly \emph{non-convex}, and many research suggest that, for over-parameterized models, the local minima are often \emph{degenerate} and {are \emph{non-singletons}} \citep{sagun2016eigenvalues,safran2016quality,freeman2017topology,sagun2017empirical,venturi2018spurious,liang2018understanding,draxler2018essentially,garipov2018loss,liang2018adding,nguyen2019connected,kawaguchi2020elimination,kuditipudi2019explaining,lin2024exploring}. 
In emerging domains like Large Language Model, the number of parameters is of order billions and the scaling laws implies that the over-parameterization phenomenon will be increasingly more prominent.
These important cases fall outside the scope of existing studies.
\vspace{-1mm}
\paragraph{\logPLmeasure\ measures} We aim to address more complex structural possibilities within the uni-modal framework.
% In the rest of the paper, we use $\optimalset{}$ to denote the collection of all local minima of the potential $V$.
% We take a step toward identifying assumptions that not only enable fast sub-exponential convergence of Langevin dynamics but also allow the corresponding optimal set to encompass more complex topological and geometric structures, hoping to broaden applicability of theoretical guarantees.
% In brief, we need assumptions that 
To define the uni-modal measure class of interest, we state the necessary assumptions.
% Define the local Polyak-Łojasiewicz (P\L) condition \citep{lojasiewicz1963topological,polyak1963gradient}.
% In this work, we consider the case where the potential function $V$ satisfies the following local Polyak-Łojasiewicz (P\L) inequality .
% \vspace{-2mm}
% \begin{my_definition}[Locally Polyak-Łojasiewicz (P\L) function] \label{def_local_pl}
%     A function $V \in\mathcal{C}^1(\R^d, \R)$ is said to be locally P\L, if there exists some constant $\PLconstant > 0$ such that for any connected component $\optimalset{}'$ in the collection of $V$'s local minima, there exists an open neighborhood, $\mathcal{N}(\optimalset{}') \supseteq \optimalset{}'$, such that
%     \begin{equation}\label{PL condition}
%         \setlength{\abovedisplayskip}{5pt}
%         \setlength{\belowdisplayskip}{5pt}
%         \forall x\in\mathcal{N}(\optimalset{}'),\ |\nabla V(x)|^2 \geq \PLconstant \left(V(x) - \min_{x\in\mathcal{N}(\optimalset{}')} V(x)\right).
%     \end{equation}
% \end{my_definition}
\begin{assumption} \label{ass_PL}
    Let $\optimalset{}$ the collection of all local minima of the potential $V\in \mathcal{C}^2$.
    % The potential $V\in \mathcal{C}^2$.
    For every connected component $\optimalset{}'$ in $\optimalset{}$, there exists an open neighborhood $\mathcal{N}(\optimalset{}')\supset \optimalset{}'$ such that, in $\mathcal{N}(\optimalset{}')$, $V\in\mathcal{C}^3$ and moreover it is locally P\L \citep{lojasiewicz1963topological,polyak1963gradient}:
    \begin{equation}\label{PL condition}
        \setlength{\abovedisplayskip}{2pt}
        \setlength{\belowdisplayskip}{5pt}
        \forall x\in\mathcal{N}(\optimalset{}'),\ |\nabla V(x)|^2 \geq \PLconstant \left(V(x) - \min_{x\in\mathcal{N}(\optimalset{}')} V(x)\right).
    \end{equation}
    % $\forall x\in\mathcal{N}(\optimalset{}'),\ |\nabla V(x)|^2 \geq \PLconstant \left(V(x) - \displaystyle\min_{x\in\mathcal{N}(\optimalset{}')} V(x)\right)$
    % The potential $V\in \mathcal{C}^2$ is locally P\L, and $V \in \mathcal{C}^3$ in every neighborhood $\mathcal{N}(\optimalset{}')$.
\end{assumption}
\noindent 
% The above assumption ensures ``sharp boundaries'' of the local optimal sets, meaning that the landscape within $\mathcal{N}(\optimalset{}')$ is not overly flat: To better see this intuition, \citet{rebjock2024fast} prove that $V$ exhibits quadratic growth in $\mathcal{N}(\optimalset{}')$, thereby enforcing a curvature that prevents flatness in this region.
The above assumption ensures
``sharp boundaries'' of the local optimal sets, i.e. the landscape within $\mathcal{N}(\optimalset{}')$ is not overly flat: \citet{rebjock2024fast} prove that the local P\L\ condition implies local quadratic growth, thereby enforcing a curvature that prevents flatness in this region.

\noindent Next, we need to exclude the possibility of saddle points so as to ensure the uni-modality. 
% since otherwise the Eyring-Kramers law implies convergence in exponential time-scale. 
In words, the following assumption states that a critical point of $V$ is either a local minimum or maximum\footnote{If we know a priori that $\GibbsMeasure$ is uni-modal, \Cref{ass_no_saddle} can be relaxed to ``for all $x\in \mathbb{R}^d \backslash \mathcal{N}(\optimalset{})$, if $\nabla V(x) = 0$, $\lambda_{\min}(\nabla^2 V(x)) < 0$''. However, this relaxed assumption is not sufficient to guarantee the uni-modality of $\GibbsMeasure$.}.
\vspace{-1mm}
\begin{assumption} \label{ass_no_saddle}
    Let $\mathcal{N}(\optimalset{})$ be the union of all the neighborhoods $\mathcal{N}(\optimalset{}')$ defined in \Cref{ass_PL}. For any $x\in \mathbb{R}^d \backslash \mathcal{N}(\optimalset{})$, if $\nabla V(x) = 0$, one has $\nabla^2 V(x) \prec 0$.
\end{assumption}
\vspace{-1mm}
\noindent We further need all the local minima to be within a compact set. This is a technical assumption and we believe it can relaxed to the coercivity of $V$. The latter is typically necessary for $Z_\epsilon < \infty$.
\vspace{-1mm}
\renewcommand{\theassumption}{3'}
\begin{assumption} \label{ass_coercivity}
    $V$ is coercive and all local minima of $V$ are contained in a compact set.
    % \footnote{This assumption is implied by the stronger \Cref{ass_growth_V} required by our later analysis. 
    % However, to ensure that $\GibbsMeasure$ is uni-modal, it is sufficient.
    % }.
    % \textcolor{red}{Make this Assumption 3'.}
\end{assumption}
\renewcommand{\theassumption}{\arabic{assumption}}
\addtocounter{assumption}{-1}
\vspace{-1mm}
\noindent 
% \textcolor{red}{highlight that the optimal set is a manifold.}
We call $V$ a {\PLcirc\ function} if it satisfies the above conditions, and refer to the Gibbs measure $\GibbsMeasure$ as a {\logPLmeasure}\ measure\footnote{The superscript $^{\circ}$ highlights that the mode of $\GibbsMeasure$ can be a $d$-sphere, a representative embedding submanifold.}. 
Note that if $V$ is globally P\L\ and coercive, {the above assumptions, apart from the regularity ones, follow directly.}
However, 
% unlike such a global P\L\ condition, 
our assumptions allow for the existence of local maxima, making them strictly weaker.
This distribution class is of interest for the following reasons:

\vspace{-1mm}
\begin{itemize}[leftmargin=*]
    \setlength\itemsep{0em}
    \item (Relevance to crucial problems.) The local P\L\ inequality is established for a class of over-parameterized neural networks \citep{oymak2020toward,liu2022loss}, 
    % \red{linear neural network?}
    % robust control \citep{fazel2018global}, \red{ref!}, 
    and P\L\ functions is an important class in the optimization literature \citep{karimi2016linear,yang2020global,rebjock2024fast}.
    \item (Connectivity of optimal set.) We prove in \Cref{lemma_unimodal} that, when the ambient dimension $d\geq 2$, \Cref{ass_PL,ass_no_saddle,ass_coercivity} together imply that the collection of all local minima has only one connected component. Hence, \logPLmeasure\ measures are uni-modal. Our proof is built on a generalized version of the famous Mountain Passing Theorem \citep{katriel1994mountain}.
    \item (Optimal set with pertinent structures.) Built on the connectivity result above, we prove that the optimal set $\optimalset{}$ is a \emph{$C^2$ embedding submanifold} of the ambient space $\R^d$ \emph{without boundary}. This class of optimal set is highly pertinent to the machine learning community \citep{cooper2018loss,cooper2020critical,fehrman2020convergence,li2022happens,wojtowytsch2024stochastic,levin2024effect}\footnote{Existing papers assume this structure of $\optimalset{}$ without proof, but we derive this result from our assumptions on $V$.}.
    In particular, the optimal set $\optimalset{}$ can be \emph{non-contractible}, thus \emph{topologically} different from convex sets.
    A simple example satisfying all assumptions is $V(x) = \|x\|^3/3 - \|x\|^2/2$, whose global optimal set is $\|x\| = 1$, forming a non-convex, non-contractible embedding submanifold.
    % See its plot in \Cref{fig:pl function}.
    % The embedding structure of the submanifold $\optimalset{}$ is crucial to define the Laplacian-Beltrami operator.
    % Please find an example in \Cref{section_PL_example}.
\end{itemize}

% \begin{figure}
%     \centering
%     \includegraphics[width=0.5\linewidth]{Arxiv/pl_3D.png}
%     \caption{Caption}
%     \label{fig:enter-label}
% \end{figure}
% Given the deep connection of P\L\ function to machine learning applications and its 
\noindent Consequently, we believe understanding the convergence of the Langevin dynamics towards a \logPLmeasure\ measure can provide a rich template for studying crucial problems like deep learning.

Further, to prove a fast sub-exponential convergence, w.r.t. ${1}/{\epsilon}$, of the Langevin dynamics, we make some technical assumptions, complied in \Cref{section_assumptions} for the ease of reference.

\vspace{-1mm}
\paragraph{Our result.} We prove that a \logPLmeasure\ measure, while far from being log-concave, possesses a 
PI constant that is \emph{non-asymptotically} lower bounded by a \emph{temperature-independent} constant, for a sufficiently small $\epsilon$. 
The applicable temperature region depends on the geometric structure of the optimal set $\optimalset{}$.
Our result is briefly summarized as follows, and formally stated in \Cref{thm_main}.
% Here $\epsilon_0$ is some constant solely depends on the potential $V$. 
% \red{This is possible since $\optimalset{}$ is connected.}
\begin{theorem}[informal] \label{theorem_informal}
    Suppose that the potential $V$ satisfies \Cref{ass_PL,ass_no_saddle,ass_coercivity} and some additional regularity assumptions in \Cref{section_assumptions}.
    Consider the case where $\optimalset{}$ is not a singleton.
    When $\epsilon$ is sufficiently small,
    the \Poincare\ constant (\ref{eqn_poincare_inequality}) of the measure $\GibbsMeasure$ satisfies $\PIconstant_{\GibbsMeasure} = \Omega(\Eigenvalue(\optimalsetinterior{}))$.
    Here $\Eigenvalue(\optimalsetinterior{}) > 0$ denotes the first non-trivial eigenvalue of the Laplacian-Beltrami operator on $\optimalsetinterior{}$. 
    % The Gibbs measure $\GibbsMeasure$ satisfies the \Poincare\ inequality (\ref{eqn_poincare_inequality}) with parameter $\PIconstant_{\GibbsMeasure} \geq \Omega(\NeummanEigenvalue(\optimalsetinterior{}))$  Here $\NeummanEigenvalue(\optimalsetinterior{})$ denotes the Neumann eigenvalue of the Laplacian operator on the optimal set $\optimalsetinterior{}$ in case ($\CIRCLE$), and denotes the corresponding eigenvalue of the Laplacian-Beltrami operator in case ($\Circle$).
    % It is known that $0<\Eigenvalue(\optimalsetinterior{})<\infty$.
    % and $\NeummanEigenvalue(\optimalsetinterior{}) > 0$ since $\partial \optimalsetinterior{}$ is Lipschitz.
    % \red{This statement should be revised, depending if we can include $\alpha=2$.}
\end{theorem}
A direct consequence of the result above is a quantative characterization of the convergence behavior of Langevin dynamics and its discrete-time implementation in the low temperature region, such as the Langevin Monte Carlo (LMC) \citep[Theorem 7]{chewi2024analysis}, for a \logPLmeasure\ target measure. 
% With a little bit effort, we can also show that the second-order SME in \citep{li2017stochastic} approximates the noisy gradient descent uniform-in-time on \PLcirc\ functions. In contrast, previous uniform-in-time approximations are only established on strongly-convex potentials \citep{li2023uniform}.
% Moreover, since we assume $\optimalset{}$ has non-empty interior, we have $\GibbsMeasure$ weakly converges to the uniform distribution $\mu_{\optimalset{}}$ over $\optimalset{}$ \citep{hwang1980laplace}.
% Consequently, ULA provides an oracle for sampling from $\mu_{\optimalset{}}$ with computational complexity quantified by the \Poincare\ constant, which can potentially be used for volume estimation of non-convex domains.
% This provides an extension of the volume estimation techniques to the non-convex domain.
% and hence we can extend the volume estimation technique developed for convex body to the non-convex domain case \citep{lovasz2007geometry}.}
Our result also represents a significant step toward proving the stronger LSI, as the PI constant is often a crucial intermediary for estimating the LSI constant \citep[Theorem 1.2]{cattiaux2010note}.

Furthermore, to the best of our knowledge, {\emph{in the low temperature region}}, our work marks the first attempt to explore the behavior of Langevin dynamics in general non-convex landscapes with \emph{non-isolated minimizers}.
% \red{See a discussion in \Cref{section_discussion_morse}.}
In addition, technique-wise, we are the first to connect the \Poincare\ constant of a measure on $\R^d$ with the stability of the Laplacian-Beltrami eigenvalue on the optimal solution manifold, offering a novel perspective.

\vspace{-2mm}
\paragraph{Our proof strategy.} Our proof is split into two steps. 
Let $U$ be a neighborhood of the optimal set $\optimalset{}$.
First, we reduce the PI constant for a measure supported on $\R^d$ to an eigenvalue problem of the Laplacian operator on $U$. 
Seeing $U$ as an expansion of $\optimalset{}$ (since $\optimalset{}\subset U$), we then relate the eigenvalues on $U$ and $\optimalset{}$ through a stability analysis. A more technical summary is as follows.
\vspace{-2mm}
\begin{enumerate}[leftmargin=*]
    \setlength\itemsep{0.2em}
    \item First, we partition the domain $\R^d$ into a collection of subdomains, allowing us to apply existing Lyapunov methods in \citep{AOP} for establishing the PI constant, as stated in \Cref{Thm: lyapunov method}. This reduces the estimation of the PI constant $\PIconstant_\GibbsMeasure$ to the estimation of the Neumann eigenvalue of the Laplacian operator on a domain \(U = \optimalset{\sqrt{C\epsilon}}
    \).
    Here $C$ is a constant independent of $\epsilon$, and $\optimalset{\eta} := \{x\in\R^d: \text{dist}(x, \optimalset{}) \leq \eta\}$. See the precise statement in \Cref{corollary_reduce_PI_to_Neumann_eigenvalue}.
    % ==========================================
    \item Second, we establish a temperature-independent lower bound for the Neumann eigenvalue on $U$. 
    Since $\optimalset{}$ is a $\mathcal{C}^2$ embedding submanifold, $\optimalset{\eta}$ matches the \emph{tubular neighborhood} of $\optimalset{}$. By the tubular neighborhood theorem \citep{MilnorStasheff1974}, up to a diffeomorphism, we can decompose the uniform distribution on $U$ as a pair of decoupled distributions along the tangent and the normal directions respectively. With the tensorization property of the \Poincare\ inequality \citep{Bakry2013AnalysisAG}, we show that the Neumann eigenvalue on $U$ is determined by the first non-trivial eigenvalue of the Laplacian-Beltrami operator on $\optimalset{}$, when $\epsilon$ is sufficiently small.
    Here, the Laplacian-Beltrami operator is defined based on the aforementioned embedding structure of the submanifold $\optimalset{}$.
    % The embedding structure of the submanifold $\optimalset{}$ is crucial to define the Laplacian-Beltrami operator.
    % ==========================================
\end{enumerate}
\vspace{-3mm}
\paragraph{Summary of contributions}
\begin{itemize}[leftmargin=*]
    \setlength\itemsep{0em}
    \item We identify the \PLcirc\ class as a rich and suitable template for studying the (sub-exponential time) convergence behavior of stochastic dynamics on potentials that admit non-isolated minima, a phenomenon empirically observed in crucial applications like deep learning in the over-parameterized regime.
    We prove that all local minima of a \PLcirc\ function are connected and its global optimal set $\optimalset{}$ formulates a compact $\mathcal{C}^2$ embedding submanifold of $\R^d$ without boundary.
    % propose the \PLcirc\ class, characterize its optimal set
    \item Built on the above characterization of the optimal set $\optimalset{}$, we show that the \Poincare\ constant of the  Gibbs measure $\GibbsMeasure$ (\ref{eqn_gibbs}) is lower bounded by the first non-trivial eigenvalue of the Laplacian-Beltrami operator on $\optimalset{}$, when $\optimalset{}$ is non-singleton. This eigenvalue is temperature-independent. As a direct consequence, the Langevin dynamics (\ref{eqn_langevin_dynamics}) converges to $\mu_\epsilon$ at a rate of $\tilde{\mathcal{O}}(1/\epsilon)$.
\end{itemize}
\vspace{-3mm}
% FINDME
\paragraph{Relation to the sampling literature} To better position our work within the literature, we restate our primary focus: understanding the dependence of the PI constant $\PIconstant_\GibbsMeasure$ on $\epsilon$, which directly impacts the convergence rate of the Langevin dynamics. Many excellent works in the sampling literature focus on analyzing the convergence of the discretized algorithms like LMC in the \emph{constant-temperature ($\epsilon=1$)} region, for (strongly) convex or nonconvex potentials under weak conditions, e.g. a joint of dissipativity and tail growth conditions \citep{erdogdu2021convergence}. We do not directly compare with them here since we focus on different temperature regions. 

\section{Preliminaries and Assumptions}\label{Preliminaries and Assumptions}
\subsection{\Poincare\ Inequality}
\begin{my_definition}[\Poincare-Wirtinger Inequality] \label{definition_PI}
    A probability measure $\mu$ with support $\Omega\subseteq\R^d$ satisfies the Poincaré inequality with parameter $\PIconstant_\mu$, or shortly PI$(\rho_\mu)$, if one has for any $f\in \sobolevspace{1,2}(\mu)$ 
    \begin{equation} \label{eqn_poincare_inequality}
        \setlength{\abovedisplayskip}{5pt}
        \setlength{\belowdisplayskip}{5pt}
        \left\{\VAR_\mu(f) := \int_{\Omega} \left(f - \int_{\Omega} f \ud \mu\right)^2 \ud \mu\right\} \leq \frac{1}{\PIconstant_\mu} \left\{\Dirichletform_\mu(f) := \int_{\Omega} |\nabla f|^2 \ud \mu\right\},
    \end{equation}
    where $\PIconstant_\mu$ is called the \emph{PI constant} and $\sobolevspace{1,2}(\mu)$ denotes the Sobolev space weighted by $\mu$.
\end{my_definition}
\vspace{-2mm}
Here $\VAR_\mu(f)$ and $\Dirichletform_\mu(f)$ are the variance and the Dirichlet energy of the test function $f\in\sobolevspace{1,2}(\mu)$ w.r.t. $\mu$.
For the Gibbs measure $\GibbsMeasure$, we have $\Omega = \R^d$.

% Poincare inequality

% For example, let us consider famous Bakry-\'Emery\ criterion, which relies on the convexity of the potential function.
% \begin{theorem}
% (Bakry-\'Emery\ criterion, [], Proposition 4.8.1, Corollary 4.8.2) Let $V: \Omega \subset \R^d \rightarrow \R$ be a potential function with Gibbs measure $\mu_{\epsilon}$ on $\Omega$ and assume that $\nabla^2 V(x) \geq \lambda > 0$ for all $x \in \R^d$. Then $\mu_{\epsilon}$ satisfies PI() with 
% \begin{equation*}
% \rho \geq \frac{\lambda}{\epsilon}.
% \end{equation*}
% \end{theorem}

\subsection{The Lyapunov Function Approach and the Perturbation Principle}
% Let us start with explaining the Lyapunov approach for deducing a $PI$. The central notation for the Lyapunov approach is the Lyapunov function, we use the same definition as in Definition 3.7 in \cite{AOP}, which rigorously define the Lyapunov function on the domain $\Omega$.
% \begin{my_definition}\label{Definition:lyaopunov function}
% (Lyapunov function for Poinc\'are inequality). Let $V:\Omega \rightarrow \R$ be a Hamiltonian with Gibbs measure $\mu_{\epsilon}(dx) = \frac{1}{Z_{\epsilon}} \exp\{-V/\epsilon\}dx$. Then $W: \Omega \rightarrow [0,\infty)$ is a $Lyapunov\ function$ for $V$ provided that:
% \begin{enumerate}
%     \item There exist a domain $U \subset \Omega$ and constants $b > 0$ and $\lambda > 0$ such that:
%     \begin{equation}\label{IEq: Lyapunov inequality}
%     \epsilon^{-1} \generator W \leq - \sigma W + b 1_U, \ \ \ a.e.\ \text{in}\ \Omega.
%     \end{equation}

%     \item $W$ satisfies Neumann boundary conditions on $\Omega$ such that the integration by parts formula holds
%     \begin{equation}\label{Eq: IBP}
%     \forall f \in H^1(\mu|_{\Omega}): \int_{\Omega}f(-\generator W)d\mu = \epsilon \int_{\Omega} \langle \nabla f, \nabla W \rangle d\mu.
%     \end{equation}
% \end{enumerate}
% \end{my_definition}
% In this section, we review the Lyapunov function approach from \citep{AOP}.
% which reduces the \Poincare\ constant estimation of the $\GibbsMeasure$ on $\R^d$ to the same estimation, but on a compact set $U$.
% To present the lemma that our work is built on, 
Define the truncated Gibbs measure on a given domain $U\subset \R^d$ as
\begin{equation} \label{eqn_truncated_gibbs}
\setlength{\abovedisplayskip}{5pt}
            \setlength{\belowdisplayskip}{5pt}
\mu_{\epsilon, U}(dx) = \frac{1_U}{Z_{\epsilon, U}} \exp{(- \frac{V(x)}{\epsilon})}dx, \ \ \ \text{with} \ \ 
Z_{\epsilon, U} = \int_{U} \exp(- \frac{V(x)}{\epsilon})dx.
\end{equation}
% \begin{my_definition}[Truncated Gibbs measure on $U$] \label{definition_truncated_gibbs}
% For a given domain $U \subset \mathbb{R}^d$, the truncated Gibbs measure $\mu_{\epsilon,U}$ is obtained from the Gibbs measure $\mu_{\epsilon}$ by restriction to the domain $U$, that is
% \begin{equation*}
% \mu_{\epsilon, U}(dx) = \frac{1_U}{Z_{\epsilon, U}} \exp{(- \frac{V(x)}{\epsilon})}dx, \ \ \ \text{with} \ \ 
% Z_{\epsilon, U} = \int_{U} \exp(- \frac{V(x)}{\epsilon})dx.
% \end{equation*}
% \end{my_definition}
The next statement shows that a Lyapunov function and the PI for the truncated measure $\mu_{\epsilon, U}$ can be combined to get the PI for the original Gibbs measure. Our work is built on this framework.
\begin{my_proposition}\label{Thm: lyapunov method}\citep[Theorem 3.8]{AOP}
        Let $\generator := -\nabla V \cdot \nabla + \epsilon\ \Delta$ be the infinitesimal generator associated with the Langevin dynamics in \cref{eqn_langevin_dynamics}.
        A function $\mathcal{W}:\R^d \rightarrow[1, \infty)$ is a \emph{Lyapunov function} for $\generator$ if there exists $U\subseteq\R^d$, $b>0$, $\sigma > 0$, such that
        \begin{equation} \label{eqn:Lyapunov}
        \setlength{\abovedisplayskip}{5pt}
            \setlength{\belowdisplayskip}{5pt}
            \forall x\in\R^d,\ \epsilon^{-1} \generator \mathcal{W}(x) \leq -\sigma \mathcal{W}(x) + b 1_{U}(x).
        \end{equation}
        Given the existence of such a Lyapunov function $\mathcal{W}$, if one further has that the truncated Gibbs measure $\mu_{\epsilon,U}$ satisfies PI with constant $\PIconstant_{\epsilon,U} > 0$, the Gibbs measure $\mu_\epsilon$ satisfies PI with constant
        \begin{equation} \label{eqn_estimate_poincare_constant_with_subdomain}
            \setlength{\abovedisplayskip}{5pt}
            \setlength{\belowdisplayskip}{5pt}
            \rho_\epsilon \geq \frac{\sigma}{b + \rho_{\epsilon,U}}\rho_{\epsilon,U}.
        \end{equation}
        % where $\sigma$ and $b$ are the constants corresponding to the Lyapunov function $W$.
        \end{my_proposition}
        \vspace{-0.2cm}
Following \citet{AOP}, we adopt $\mathcal{W}(x) = \exp\big(\frac{1}{2\epsilon} V \big)$ as the Lyapunov function throughout this work. This function satisfies $W(x) \geq 1$ since we assume WLOG $V^* = 0$.
The only remaining argument is to establish the condition in \cref{eqn:Lyapunov}. To be more precise, we need to find two constants $\sigma > 0, b > 0$ and some set $U \subset \R^d$ such that
\begin{equation}\label{eqn_requirement_Lyapunov}
    \setlength{\abovedisplayskip}{5pt}
    \setlength{\belowdisplayskip}{5pt}
    \frac{\generator \mathcal{W}}{\epsilon \mathcal{W}} = \frac{\Delta V}{2\epsilon} - \frac{|\nabla V|^2}{4\epsilon^2} {\leq} - \sigma + b1_U.
\end{equation}
We will find these two constants in \Cref{lemma_Lyapunov_alpha_PL}. 
% The construction of $U$ is the focus of \Cref{section_stability_neumann_eigenvalue}.
% \red{Should mention somewhere that integral by parts holds.}
% In nonconvex cases, we can also consider the $Holley-Stroock\ perturbation\ principle$.
In addition to the above Lyapunov function framework, the following standard perturbation principle will also be helpful to us.
\begin{my_proposition}[Holley-Stroock perturbation principle]\label{theorem_perturbation_principle}
% Let $V$ be a Hamiltonian with Gibbs measure $\mu(dx) = \frac{1_U}{Z_{\mu}} \exp\{-V/\epsilon\}$ on $U$. 
Let $V$ and $\tilde V$ be two potential functions defined on a domain $U$.
If the truncated Gibbs measures, defined in \cref{eqn_truncated_gibbs}, with energies $V$ and $\tilde V$ satisfy PI$(\PIconstant)$ and PI$(\tilde \PIconstant)$ respectively, one has
% \begin{equation*}
$\rho \geq \exp \big\{-\frac{1}{\epsilon}\big(\sup_{x \in U}(V - \tilde{V}) - \inf_{x \in U}(V - \tilde{V})\big)\big\}\tilde \rho$.
% \end{equation*}
% where $\text{osc}_U(V - \tilde{V}):= \sup_{x \in U}(V - \tilde{V}) - \inf_{x \in U}(V - \tilde{V})$.
\end{my_proposition}

\subsection{Properties of a $\mathcal{C}^2$ Embedding Submanifold} \label{subsection: submanifold}
The Laplacian-Beltrami operator on the optimal set $\optimalset{}$ is crucial to our analysis.
The definition of this operator is built on a pullback metric induced by the embedding structure of $\optimalset{}$, outlined as follows. 
Note that in the rest of the paper, we use $k$ to denote the dimension of the manifold $\optimalset{}$ and focus on the case $k \geq 1$.
For $k=0$, $\optimalset{}$ becomes a singleton, which is not the focus of our work.

\begin{my_definition}[Embedding submanifold in $\R^d$]\label{def_submanifold}
Consider a $\mathcal{C}^2$ manifold $M$ such that $M \subseteq \R^d$. If the including map $i_M: M \rightarrow \R^d$ is $\mathcal{C}^2$ and satisfies following two conditions:
% \begin{itemize}
	(1) The tangent map {$Di_M(x)$} has rank equal to \text{dim} $M$ for all $x \in M$;
	(2) $i_M$ is a homeomorphism of $M$ onto its image $i_M(M) \subset \R^d$, where $i_M(M)$ inherits the subspace topology from $\R^d$.
% \end{itemize}
We say that the including map $i_M$ is an embedding and  $M$ is a $\mathcal{C}^2$ embedding submanifold of $\R^d$.
\end{my_definition}

% \begin{my_remark}\label{Remark: local chart of embedding}	
If $M$ is a $k$-dimensional embedding submanifold of $\R^d$, the including map can be represented using the local coordinates of $M$ as follows: Assume that $\{\Gamma_i, \phi_i\}_{i \in \Lambda}$ is the maximal atlas of $M$. For $u = (u^1,...,u^k) \in \Gamma_i \subset \R^k$, the including map $i_M: M \rightarrow \R^d$ can be written as
	\begin{equation}\label{embedding structure}
    \setlength{\abovedisplayskip}{5pt}
            \setlength{\belowdisplayskip}{5pt}
        x^j = m^i_j(u^1, u^2,...,u^k), j \in \{1, \ldots, d\},
	% \left\{
 %    \begin{aligned}
 %    x^1 & = m^i_1(u^1, u^2,...,u^k), \\
	% x^2 & = m^i_2(u^1, u^2,...,u^k), \\
	% & \dots, \\
 %    x^d & = m^i_d(u^1, u^2,...,u^k), \ \ \ \ (u^1, u^2,...,u^k) \in \Gamma_i \subset \R^k,
	% \end{aligned}
	% \right.
	\end{equation}
	where $m^i_j: \Gamma_i \subset \R^k \rightarrow \R,\ j = 1,...,d$ are $\mathcal{C}^2$ coordinate functions. We also denote this embedding structure as $\mathcal{M}^i(u) = (m^i_1(u),...,m^i_d(u))$ on a local chart $(\Gamma_i, \phi_i)$.    
% \end{my_remark}
\begin{my_remark}
    WLOG, we assume that there is only one local chart $(\Gamma, \phi)$ in the rest of the paper, since we can always extend local results to a global one by the standard technique of partition of unity  \cite[Chapter 13]{tu2010introduction}. 
    We write the corresponding embedding structure as $\mathcal{M}$, omitting the superscript.
    Any non-trivial differences encountered in related proofs will be explicitly highlighted.
    % We will emphasize related proofs if we meet any non-trivial differences. 
\end{my_remark}
Given the above embedding structure, the embedding submanifold $M$ naturally inherits Riemannian structures from the ambient space $\R^d$. In the following, we describe the first and second fundamental forms on $M$. 
The reader can find more details about these structures in \Cref{Rie structure on local chart}.
% In the following, we describe the first fundamental form on $M$. This quantity is necessary for defining the Laplacian-Beltrami operator. 

% these structures of $\optimalset{}$ as follows. 
% \begin{itemize}
     
     \paragraph{The first fundamental form (or Riemannian metric).} We define the Riemannian metric $g_{M}$ on $M$ as the pullback metric induced by the including map $i_{M}: M \hookrightarrow \R^d$, i.e. $g_{M} = i^{\ast}_{M}(g_E)$, where $g_E$ is the standard Riemannian metric on $\R^d$ and $i_{M}^{\ast}$ is the pullback map associated with $i_{M}$.
     Now we can say that $(M, g_{M})$ is a $k$-dimensional Riemannian submanifold on $\R^d$ and the including map
     $ i_{M}: (M, g_{M}) \hookrightarrow (\R^d, g_E)$
     is a Riemannian embedding. Based on this Riemannian metric, on the local chart $(\Gamma, \phi)$, we can define the \emph{Laplacian-Beltrami operator} $\Delta_{g_{M}}$ as
     \begin{equation}\label{laplacian-betrami}
     \setlength{\abovedisplayskip}{5pt}
            \setlength{\belowdisplayskip}{5pt}
     -\Delta_{g_{M}} = - \frac{1}{\sqrt{\text{det}(g_{M})}}
     \sum_{i,j=1}^k \frac{\partial}{\partial u^i} \Big( \sqrt{\text{det}(g_{M})} g^{ij} \frac{\partial}{\partial u^j} \Big) \ \ \ \ u \in \Gamma,
	\end{equation}
    and the standard volume form $d\M$ as
    % \begin{equation*}
    $d\M(u) = \sqrt{\det(g_{M})}\vert du^1 \wedge ... \wedge du^k\vert, \ u \in \Gamma,$
    % \end{equation*}
    where $\text{det}(g_{M})$ is the determinant of the matrix $g_{M} = (g_{ij})$, and $(g^{ij})$ is the inverse matrix of $g_{M}$. 

% \end{itemize}
\paragraph{The second fundamental form.} 
Recall that $\M(u)$ is the embedding structure defined above. Let $\N_{k+1},...,\N_d: \Gamma \rightarrow \R^d$ be $d-k$ normal vectors on $M$ which are orthogonal to each other.
Define the matrix $G(l) = [G_{ij}(l) ]$ with $G_{ij}(l) = - \frac{\partial^2 \M(u)}{\partial u^i \partial u^j} \cdot \N_l(u), l = k+1,...,d$, for $u \in \Gamma \subset \R^k$.
With this notation, we can define the second fundamental form of the manifold $\mathcal{M}$ by $\Pi = - \sum_{l=k+1}^d \bigg\{ r^l \sum_{i,j=1}^k G_{ij}(l) du^idu^j \bigg\}$, for some small $(r^l)_{l=k+1}^d$.
We further define a matrix $\tilde G(l) = [G_j^i(l)]$ with $G_j^i(l) = \sum_{s = 1}^k g^{is} G_{sj}(l), l = k+1,...,d$, which will be useful to our presentation.

% $\mathcal{N}_l$ -> $G_{ij}(l)$ -> $G_j^i(l)$ -> $\tilde G(l)$ -> $\Pi$

% Due to space limitation, we elaborate on the second fundamental form in \Cref{Rie structure on local chart}, which will be used in \Cref{ass_optimalset_interior}.

\subsection{The Eigenvalue Problems of Differential Operators} \label{section_eigenvalue}
% \paragraph{Neumann Eigenvalue}
In the next, we introduce the eigenvalue problem of the Laplacian operator with Neumann boundary condition on a compact set $\Omega$.  
\begin{my_definition}[Neumann eigenvalue] \label{definition_Neumann_eigenvalue} Consider the eigenvalue problem for the Laplacian operator on a closed domain $\Omega$, subject to the Neumann boundary condition
\begin{equation*}
\setlength{\abovedisplayskip}{5pt}
            \setlength{\belowdisplayskip}{5pt}
\begin{aligned}
 - \Delta u = \lambda u,  x \in \mathrm{int}\ \Omega \text{ and } {\partial u}/{\partial \nu} = 0,  x \in \partial \Omega,
\end{aligned}
\end{equation*}
where $\nu$ is the outward normal vector to $\partial \Omega$ and $u \in \sobolevspace{1,2}(\Omega)$. The Neumann eigenvalue {$\NeummanEigenvalue(\Omega)$} is defined to be the minimum non-zero eigenvalue $\lambda$ to the above problem.
\end{my_definition}
\noindent 
Recall the \Poincare\ inequality in \Cref{definition_PI}.
It is known that $\NeummanEigenvalue(\Omega)$ matches the {best} \Poincare\ constant for the Lebesgue measure on $\Omega$. {We can use min-max formulation for Neumann eigenvalue of Laplacian operator to derive this fact, see \cite[Theorem 4.5.1]{Davies_1995}}, i.e., it admits the following variational formulation:
% \begin{my_remark}[Min-max formulation for Neumann eigenvalue] \label{remark_Neumann_eigenvalue_min_max}
    % Consider the first non-zero Neumann eigenvalue defined above. It admits the following variational formulation
    \begin{equation} \label{remark_Neumann_eigenvalue_min_max}
        \NeummanEigenvalue(\Omega) = \inf_{\stackrel{L\subseteq \sobolevspace{1,2}(\Omega)}{\text{dim}(L) = 2}} \sup_{u \in L} \left\{ \frac{|\nabla u|^2_{L^2(\Omega)}}{|u|^2_{L^2(\Omega)}}\right\}  = \min \left\{ \frac{|\nabla u|^2_{L^2(\Omega)}}{|u|^2_{L^2(\Omega)}}: u\in W^{1,2}(\Omega)\backslash \{0\}, 
        \int_\Omega u(x)\ud x = 0\right\}.
        % \partial u /\partial \Vec{n} = 0 \text{ on } \partial\Omega\right\}.
    \end{equation}
    % It is known that $\PIconstant(\mu_0) = \NeummanEigenvalue(\optimalset{})$.
    % We will show that $\lim_{\epsilon\rightarrow0} \PIconstant(\mu_\epsilon) = \NeummanEigenvalue(\optimalset{})$.
    % Denote the uniform distribution over $\optimalset{}$ by $\mu_0 := \mathrm{Uniform}(\optimalset{})$. By the Poincaré–Wirtinger inequality, we have that $\mu_0$ satisfies $\PI{\rho_0}$ for some constant $\rho_0 > 0$ that depends only on the optimal set $\optimalset{}$. 
% \end{my_remark}
We will exploit the above formulation for $\Omega = U = \optimalset{\sqrt{C\epsilon}}$ in our analysis. 
% \paragraph{Eigenvalue of the Laplacian-Beltrami Operator}

% \subsection{Case ($\Circle$): Eigenvalue of the Laplacian-Beltrami Operator}
% Under above assumptions, we can introduce some basic Riemannian structures of submanifold $S$ in $\R^d$.

% \red{A paragraph that introduces the eigenvalue of the Laplacian-Beltrami operator.}

Next, we introduce the eigenvalue problem of the Laplacian-Beltrami operator on the compact Riemannian submanifold $M$, which strongly depends on the non-trivial metric $g_{M}$.

\begin{my_definition}[Eigenvalue of the Laplacian-Beltrami operator] \label{def_eigenvalue_laplacian_beltrami} Consider the eigenvalue problem for the
Laplacian-Beltrami operator on the Riemaniann submanifold $(M,g_{M})$ without boundary,
\begin{equation} \label{eqn_eigenvalue_LB}
\setlength{\abovedisplayskip}{5pt}
            \setlength{\belowdisplayskip}{5pt}
- \Delta_{g_{M}} u = \lambda u, \ \ \ x \in M,
\end{equation}
where $u \in \sobolevspace{1,2}(M)$ and $-\Delta_{g_{M}}$ is the Laplacian-Beltrami operator on $M$ associated with metric $g_{M}$.
The eigenvalue $\Eigenvalue(M)$ is defined to be the minimum non-zero eigenvalue $\lambda$ to the above problem.
\end{my_definition}
Define the Dirichlet energy of Laplacian-Beltrami operator on $\optimalset{}$, for $f \in W^{1,2}(M)$
\begin{equation*}
\setlength{\abovedisplayskip}{5pt}
            \setlength{\belowdisplayskip}{5pt}
Q_M(f,f) = 
% \left\{
% \begin{aligned}
% & 
\int_{M} \langle f, -\Delta_{g_{M}} f \rangle_{g_E} d\M = \int_{M} \langle df, df \rangle_{g_{M}} d\M.
% & f \in W^{1,2}(\optimalset{}), \\
% & +\infty & \text{otherwise},  
% \end{aligned}
% \right.
\end{equation*}
Here $d$ is the exterior derivative on cotangent bundle $T^{\ast}M$, which can be written as 
% \begin{equation*}
$df = \sum_{i,j=1}^k g^{ij} \frac{\partial f}{\partial u^j} \frac{\partial}{\partial u^i} = \sum_{i=1}^k g^{ij} \nabla_{u^j}f$,
% \end{equation*}
on the local chart $(\Gamma, \phi)$ by the duality between tangent bundle $TM$ and cotangent bundle $T^{\ast}M$. Then min-max theory tells us that
\begin{equation*}
\setlength{\abovedisplayskip}{5pt}
            \setlength{\belowdisplayskip}{5pt}
\Eigenvalue(M) = \inf_{\tilde{L} \subseteq \sobolevspace{1,2}(M) : \text{dim} \tilde{L} = {2}} \sup_{u \in \tilde{L}} \frac{Q_{M}(u, u)}{\int_{M} |u|^2 d\M}.
\end{equation*}
% we also use the notation $\lambda_1(\optimalset{})$ to denote $\lambda_1(-\Delta_{g_S})$ in the following. 
The reader could find more materials about this part in \cite[Chapter3]{Brard1986SpectralGD}.
In our analysis, we will take $M = \optimalset{}$ in any manifold-related content.

\subsection{Additional Regularity Assumptions} \label{section_assumptions}
For the ease of reference, we summarize the required assumptions in this subsection.

\begin{assumption}[Behavior of $V$ beyond a compact set] \label{ass_growth_V}
    Beyond a compact set, $V$ satisfies the error bound inequality, i.e. $\exists R_0>0$ such that $\forall |x|\geq R_0$\footnote{Clearly, \cref{eqn:error bound} implies \Cref{ass_coercivity}. Henceforce, we  refer to \Cref{ass_growth_V} when \Cref{ass_coercivity} is required.
},
    \begin{equation} \label{eqn:error bound}
    \setlength{\abovedisplayskip}{5pt}
            \setlength{\belowdisplayskip}{5pt}
        |\nabla V(x)| \geq \errorboundconstant\dist(x, \optimalset{}).
    \end{equation}    
    Moreover, $\Delta V := \udiv \nabla V$ grows at most polynomially beyond a compact set, i.e. $\forall |x|\geq R_0,\ |\Delta V(x) | \leq C_{g} |x|^2$.
    WLOG, we assume that $R_0$ is sufficiently large so that for all $x\in \optimalset{}$, $|x|\leq R_0$.
\end{assumption}

\begin{assumption}\label{ass_optimalset_interior}
    % One of the following conditions holds:
    % \begin{enumerate}
        % \item[$(\Circle)$] 
        % The optimal set $\optimalset{}$ is a compact $\mathcal{C}^2$ embedding submanifold without boundary. 
        Let $k$ be the dimension of $\optimalset{}$. We assume $k \geq 1$, i.e. $\optimalset{}$ is not a singleton.
        Moreover, we assume $\optimalset{}$ to have a bounded second fundamental form: On the local chart $(\Gamma, \phi)$ of $S$, 
        % \begin{equation}\label{bounded curvature}
        \(
        \sup_{k+1 \leq l \leq d}\|\tilde{G}(l)\|_{\infty} 
        % \footnote{The boundedness condition of the second fundamental form is necessary. There exist curves in the plane that do not have a bounded curvature. Let us take the curve as $(x(t), y(t))$, then its second fundamental form is $\kappa(t) = \frac{|x''(t)y'(t) - x''(t)y'(t)|}{(x'^2(t) + y'^2(t))^{\frac{3}{2}}}$. It is easy to know that it is not bounded for ``Tractrix Curve"
        % \( (x(t), y(t)) = (a \sin t, a \ln (\tan (t/2)) + a \cos t), \kappa(t) = \Big| \frac{\tan t}{a}\Big|\)
        % at the point $(a, 0)$.
        % } 
        < \infty,
        \)
        where $\tilde{G}(l), l = k+1,...,d$ are defined in the end of \Cref{subsection: submanifold}.
        
        % are the matrices that define the second fundamental form of $\optimalset{}$.
        % \item[$(\CIRCLE)$] The optimal set $\optimalset{}$ has non-empty interior and satisfies $\partial \optimalsetinterior{} \in \text{Lip}(M, \delta, s, \mathbf{V}, \mathbf{\Lambda})$.
    % \end{enumerate}
\end{assumption}
\begin{my_remark}
The boundedness condition of the second fundamental form is necessary. There exist curves that do not have a bounded curvature: Consider the ``Tractrix Curve" parameterized by $t$, defined as \( (x(t), y(t)) = (a \sin t, a \ln (\tan (t/2)) + a \cos t)\). Its second fundamental form is $\kappa(t) = \frac{|x''(t)y'(t) - x''(t)y'(t)|}{(x'^2(t) + y'^2(t))^{\frac{3}{2}}} = \Big| \frac{\tan t}{a}\Big|$, which is not bounded at the point $(a, 0)$ or equivalently $t=\frac{\pi}{2}$.
\end{my_remark}
\begin{my_remark}[Dimension of $\optimalset{}$] \label{remark_dimension_of_optimal_set}
    We prove in \Cref{lemma_manifold_structure} that $\optimalset{}$ has no boundary and hence its dimension $k$ is strictly smaller than $d$: If $k=d$, then $V$ is a constant function, which contradicts with \Cref{ass_growth_V}. Consequently, we focus on $1 \leq k \leq d-1$. For $k=0$, i.e. $\optimalset{}$ degenerates to a singleton, the PI constant under a global P\L\ condition has been recently established in \citep{Chewi2024-ch}. In this case, the PI constant is of order $\Omega(\frac{1}{\epsilon})$ since $V\in\mathcal{C}^2$ and the P\L\ condition implies that $V$ is locally strongly convex near the unique minimum.
\end{my_remark}

\begin{my_remark}[$\Eigenvalue(\optimalset{})$ is non-trivial]
    % \red{@Yungong, please add references showing that $0< \NeummanEigenvalue(\optimalset{}) < \infty$.} 
    The \Poincare\ inequality on Riemannian manifold has been well-studied. We refer readers \cite[Theorem 2.10]{Hebey1999NonlinearAO} to the case of compact Riemannian manifold, showing that $0< \Eigenvalue(\optimalset{}) < \infty$.
    % and \red{\cite{poincareIneqWithCompleteManifold} to the case of complete Riemannian manifold with some curvature and growth conditions}.
\end{my_remark}
% \begin{my_remark}[$\NeummanEigenvalue(\optimalset{})$ is non-trivial in case $(\CIRCLE)$]
%     Note that the optimal set $\optimalset{}$ is connected and has Lipschitz boundary, the \Poincare\ constant of $\optimalset{}$ exists, see \cite[Theorem 3, \S 5.6]{evans10}, so that $\lambda_1(\optimalset{}) > 0$. Moreover, we have $\lambda_1(\optimalset{}) < \infty$ by \cite[Remark 14]{burenkov2002spectral}.  
% \end{my_remark}

% \begin{my_remark}[Some incompatible configurations]
%     While our result applies on different combinations of the P\L\ index $\alpha$ and the two cases in \Cref{ass_optimalset_interior}, certain combinations are incompatible. One most notable case is $\alpha = 2$ and $\optimalset{}$ is a Lipschitz domain with interior: If $\optimalset{}$ has interior, $V \equiv 0$ on $\optimalset{}$ and hence $\nabla^2 V \equiv 0$ on $\partial \optimalset{}$ since $V \in \mathcal{C}^2$. This contradicts with the quadratic growth of $V$ which is implied by $2$-PL \citep[Proposition 2.2]{rebjock2024fast}.
%     Another interesting observation is that if we strengthen the \Holder\ continuity of $\nabla^2 V$ to Lipschitz continuity in \Cref{ass_regularity_V} for $\alpha \in (1.5, 2)$, local $2$-PL also holds. See the discussion in \citep[Remark 2.21]{rebjock2024fast}.
% \end{my_remark}

% \clearpage
% ===================================================================================================

% ===================================================================================================
\section{Step 1: Reduction to the Neumman Eigenvalue Problem} \label{section_reduction_to_neumann_eigenvalue}
% Let $U$ be any subdomain that satisfies some set inclusion relations, specified in (\ref{eqn_requirement_on_U}). 
We show that, when the temperature $\epsilon$ is sufficiently small, the \Poincare\ constant of $\GibbsMeasure$ can be lower bounded by the Neumann eigenvalue of the Laplacian operator on a closed domain $U$ in $\R^d$. 
% Here the choice of $U$ should satisfy some set inclusion relation (\ref{eqn_requirement_on_U}).
We will first list a few useful properties of the \logPLmeasure\ measures and then present our proof.

\subsection{Properties of the \logPLmeasure\ measures}
The most important property of the \logPLmeasure\ measures is its uni-modality, and the manifold characterization of its mode, 
which we state in the following. Detailed proof can be found in \Cref{section_log_PL_uni_modal}.
\begin{my_proposition}[Uni-modality] \label{lemma_unimodal}
    Under \Cref{ass_PL,ass_no_saddle,ass_growth_V}, the all local minima of the potential function $V$ are connected. Hence, $V$ has only one connected global minima set $\optimalset{}$.
    % and the Gibbs measure $\GibbsMeasure$ is single modal.
\end{my_proposition}
\noindent The proof the following corollary is built on \citep[theorem 2.16]{rebjock2024fast}. The purpose of this restatement is to explicitly exclude the possibility of boundary, which is not discussed in the previous work, and strengthen their local manifold structure to a global one using the above connectivity result. The absence of boundary significantly simplifies the eigenvalue problem (\ref{eqn_eigenvalue_LB}).
\begin{my_corollary}[Manifold structure] \label{lemma_manifold_structure}
$\optimalset{}$ is a $\mathcal{C}^2$-embedding submanifold of $\R^d$ {without boundary}.
\end{my_corollary}

We then characterize the properties of $V$ in three different regions: (1) when $x$ is close to the global minima set; (2) when $x$ is close to some local maximum; (3) otherwise.
\begin{my_lemma} \label{lemma_property_of_V}
    Under \Cref{ass_PL,ass_no_saddle,ass_growth_V}, the function $V$ satisfies the following properties:
    \begin{itemize}[leftmargin=*]
        \setlength\itemsep{0em}
        \item For any $x \in \partial \mathcal{N}(\optimalset{})$, there exists some constant $\delta_0 > 0$ such that $\mathrm{dist}(x, \optimalset{}) \geq \delta_0$.
        \item Let $X$ denote the set of all local maxima of the potential $V$. 
        % $X$ contains at most finitely many singletons. Moreover, 
        If $X\neq \emptyset$, there exists constants $R_1 > 0$ and $\mu^{-} > 0$ such that for all $x \in \mathcal{N}(X) := \{x: \mathrm{dist}(x, X) \leq R_1\}$, $\nabla^2 V(x) \preceq - \mu^{-} I_d$.
        \item For $x \notin {\mathcal{N}(X)} \cup \mathcal{N}(\optimalset{})$, there exists some constant $g_0>0$ such that $\|\nabla V(x)\| \geq g_0$.
    \end{itemize}
\end{my_lemma}

The following results are direct implications of our assumptions. The constants therein will be used in the statement of Step 1.
\begin{my_lemma} \label{lemma_useful_constants}
     Under \Cref{ass_PL,ass_no_saddle,ass_growth_V}, the function $V$ satisfies the following properties:
    \begin{itemize}[leftmargin=*]
    \setlength\itemsep{0em}
        \item Since $V \in \mathcal{C}^2$, there exists some constant $L$ such that $\|\nabla^2 V(x)\|\leq L$ for all $x \in \mathcal{N}(\optimalset{})$.
        \item Recall $R_0$ from \Cref{ass_growth_V}. Since $V \in \mathcal{C}^2$, there exists some constant $M_\Delta>0$ such that $|\Delta V(x)| \leq M_\Delta$ for $\|x\|\leq R_0$. WLOG, assume $M_\Delta \geq d \mu^{-}$. Otherwise, simply set $\mu^{-} = M_\Delta / d$.
        \item There exists a neighborhood $\mathcal{N}'(\optimalset{})$ of $\optimalset{}$ on which the error bound inequality \cref{eqn:error bound} holds with a constant $\PLconstant'$ \citep{rebjock2024fast}.
        For simplicity, we assume $\mathcal{N}'(\optimalset{}) = \mathcal{N}(\optimalset{})$, and $\PLconstant' = \PLconstant$ since otherwise we can set $\mathcal{N}(\optimalset{}) = \mathcal{N}(\optimalset{}) \cap \mathcal{N}'(\optimalset{})$ and $\PLconstant = \min\{\PLconstant, \PLconstant'\}$ in \Cref{ass_PL}. All derivations remain unchanged.
    \end{itemize} 
\end{my_lemma}

\subsection{Proof Sketch of Step 1}

Our proof is built on the Lyapunov approach described in \Cref{Thm: lyapunov method}. To meet the requirements in \cref{eqn_requirement_Lyapunov}, for any $x$ outside of a closed domain $U = \optimalset{\sqrt{C\epsilon}}$, we need (i) a lower bound for the gradient norm $||\nabla V(x)||$ and (ii) an upper bound for the Laplacian $|\Delta V(x)|$. 
% We discuss these points as follows. 
% Point (iii) will be the focus of \Cref{section_stability_neumann_eigenvalue}.

\paragraph{(i) Lower bound of gradient norm $\|\nabla V(x)\|$.}
There are four situations:
% \begin{itemize}
    (a) When $x$ is outside of a compact set, we utilize the error bound inequality in \Cref{ass_growth_V};
    (b) When $x \in \mathcal{N}(\optimalset{})$, we utilize the error bound inequality in \Cref{lemma_useful_constants};
    % \textcolor{red}{fix here. properly cite \citep{rebjock2024fast}}
    % \vspace{-7mm}
    % \begin{my_lemma}\label{Lemma:lower bound of nablaV}
    % Suppose \Cref{ass_PL,ass_no_saddle,ass_coercivity} hold. 
    % The potential function $V$ satisfies the error bound inequality in \cref{eqn:error bound} on $\mathcal{N}(\optimalset{})$, where we recall $\mathcal{N}(\optimalset{})$ in \Cref{ass_no_saddle}.
    % % \begin{equation} \label{eqn:error bound}
    % %     |\nabla V(x)| \geq \errorboundconstant\dist^{\frac{1}{\alpha-1}}(x, \optimalset{}).
    % % \end{equation}    
    % \end{my_lemma}
    (c) When $x$ is close to a local maximum, it suffices to use the trivial bound $||\nabla V(x)|| \geq 0$;
    (d) Otherwise, we utilize the third property in \Cref{lemma_property_of_V}.
% \end{itemize}

\paragraph{(ii) Upper bound of Laplacian $|\Delta V(x)|$.} 
% The point (ii) provides hint for the requirements of the subdomain $U$ in point (iii) and is the major technical novelty of Step 1. 
    To bound  the Laplacian $|\Delta V(x)|$, we partition $\R^d\backslash\optimalsetinterior{\sqrt{C\epsilon}}$ into two of subdomains $\Xi_1$ and $\Xi_2$ and sketch the treatments:
    % \begin{align}
    %     \R^d = \underbrace{\left\{x: 2R_0 \leq \|x\| \right\}}_{:=\Xi_{1}} \cup&\
    %     \underbrace{\left\{x: \sqrt{C\epsilon} \leq \dist(x, \optimalset{}) \text{ and } \|x\| \leq 2R_0\right\}}_{:=\Xi_2} \cup \optimalsetinterior{\sqrt{C\epsilon}}. \label{eqn_partition_of_R_d}
    % \end{align}
    % Here, we recall $R_0$ from \Cref{ass_growth_V}, $C$ is some constant independent of $\epsilon$ defined in \Cref{lemma_Lyapunov_alpha_PL}.
% Now we investigate how this condition is satisfied under different circumstances.
% \begin{itemize}
    For $x\in\Xi_{1} = \left\{x: 2R_0 \leq \|x\| \right\}$, we bound the Laplacian term $|\Delta V(x)|$ by utilizing the growth of $\Delta V$ from Assumption \ref{ass_growth_V}, in which $R_0$ is defined.
    {For $x\in\Xi_{2} = \left\{x: \sqrt{C\epsilon} \leq \dist(x, \optimalset{}) \text{ and } \|x\| \leq 2R_0\right\}$, there are three situations: (1) $x$ is close to some local maximum; (2) $x \in \mathcal{N}(\optimalset{})$; (3) otherwise. We treat these three situations separately.
    The power index $\frac{1}{2}$ in $\Xi_2$ is the largest value such that the positivity of $\sigma$ in \cref{eqn_requirement_Lyapunov} still holds when $x \in \mathcal{N}(\optimalset{})$.} Here $C$ is some constant independent of $\epsilon$ defined in \Cref{lemma_Lyapunov_alpha_PL}. 
    % \item The estimation on $\Xi_{i\geq 1}$ is achieved in a recursive manner: Suppose that $\epsilon$ is sufficiently small so that \Cref{ass_regularity_V} applies, i.e. $(C\epsilon)^{n_i} \leq \delta$. We can upper bound $|\Delta V(x)|$ by exploiting the facts that $\nabla^2 V(x) \equiv 0$ for $x\in\optimalset{}$ and $\nabla^2 V$ is locally \Holder\ continuous. More precisely, we have
    % \begin{equation*}
    %     \|\nabla^2 V(x) - \nabla^2 V(x^*)\| \leq L \| x - x^*\|^{\frac{2-\alpha}{\alpha - 1}} \leq L \left(C\epsilon\right)^{n_{i} \cdot \frac{2-\alpha}{\alpha - 1}}.
    % \end{equation*}
    % We can identify the index $n_{i+1}$ as the largest value such that the positivity of $\sigma$ in \cref{eqn_requirement_Lyapunov} still holds.
    % Moreover, we explicitly calculate the expression of $\{n_i\}_{i=1}^\infty$, show that $n_{i+1} > n_i$, and prove that the sequence converges to $\frac{\alpha-1}{\alpha}$. 
% \end{itemize}
Note that for each subdomain $\Xi_i, i=1,2$, there will be a corresponding value of $\sigma_i$ and we set $\sigma = \inf_{i\in\{1, 2\}} \sigma_i$. 
Besides, we estimate $b$ by restricting $x$ on $U = \optimalset{\sqrt{C\epsilon}}$. 

\subsection{Establishing the bound on $b$ and $\sigma$.}
% We start by stating the local error bound inequality of the potential function $V$, which will be used throughout our derivation. Please find its proof in \Cref{section_proof_of_error_bound}.
% bound the gradient
% \begin{my_lemma}\label{Lemma:lower bound of nablaV}
%     Suppose \Cref{ass_PL,ass_no_saddle,ass_coercivity} hold. 
%     The potential function $V$ satisfies the error bound inequality in \cref{eqn:error bound} on $\mathcal{N}(\optimalset{})$, where we recall $\mathcal{N}(\optimalset{})$ in \Cref{ass_no_saddle}.
%     % \begin{equation} \label{eqn:error bound}
%     %     |\nabla V(x)| \geq \errorboundconstant\dist^{\frac{1}{\alpha-1}}(x, \optimalset{}).
%     % \end{equation}    
% \end{my_lemma}
\noindent 

% bound Laplacian
% For the sake of a tight analysis, we bound the Laplacian term in \cref{eqn_requirement_Lyapunov} separately on different partitions of $\R^d \backslash \optimalsetinterior{(C\epsilon)^{\frac{\alpha-1}{\alpha}}}$, as previously mentioned. 

The following lemma formally states the above results, whose proof is in \Cref{proof_of_Lyapunov_alpha_PL}.
\begin{my_lemma} \label{lemma_Lyapunov_alpha_PL}
    Suppose that \Cref{ass_growth_V,ass_no_saddle,ass_optimalset_interior,ass_PL} hold.
    Define a constant $C = \frac{4M_\Delta}{\PLconstant^2}$.
    Suppose that $\epsilon$ is sufficiently small such that 
        $\epsilon \leq \min\{\frac{\errorboundconstant^2}{64C_g}, \frac{\delta_0^2}{C}, \frac{g_0^2}{4M_\Delta} \}$.
        % $\epsilon \leq \frac{\errorboundconstant^2}{64C_g}$, $\left(C\epsilon\right)^{\frac{1}{2}} \leq \delta_0$, $\frac{g_0^2}{4\epsilon^2} \geq \frac{M_\Delta}{\epsilon}$.
    % Define two sequences $\{n_i\}_{i=1}^\infty$ and $\{m_i\}_{i=1}^\infty$ such that 
    % \begin{equation}
    %     n_1 = \frac{\alpha-1}{2},\ m_i = \frac{1}{n_{i+1}}-2, \text{ and } m_i\cdot n_{i} = \frac{2n_{i+1}}{\alpha -1}.
    % \end{equation}
    % We have $m_i \geq \frac{2-\alpha}{\alpha-1}$ and $n_{i+1} > n_i$ for all $i\geq 1$, and $\lim_{i\rightarrow\infty} n_i = \frac{\alpha-1}{\alpha}$.
    % Moreover, on the subdomain $\Xi_i = \left\{x: \left(C\epsilon\right)^{n_{i+1}} \leq \dist(x, \optimalset{}) \leq \left(C\epsilon\right)^{n_i} \right\}_{i=1}^{\infty}$, we have \cref{eqn_requirement_Lyapunov} hold with $\sigma_i = \frac{\errorboundconstant^2}{8\epsilon^2} (C\epsilon )^{m_in_i + 1}$.\\
    Choose $U = \optimalset{\sqrt{C\epsilon}}$. \Cref{eqn_requirement_Lyapunov} holds with
    \begin{equation} \label{eqn_estimation_sigma_b}
    \setlength{\abovedisplayskip}{5pt}
            \setlength{\belowdisplayskip}{5pt}
        \sigma = \min \bigg\{  \frac{\errorboundconstant^2R_0^{2}}{128} \frac{1}{\epsilon^2},\  \frac{d\mu^{-}}{2\epsilon}\bigg\}, \text{ and } b = \sigma + \frac{M_\Delta}{2\epsilon}.
    \end{equation}
    
    % assumption
    % definition of \sigma_i, \Xi_i, n_i
    % limit of n_i
    % value of b
\end{my_lemma}

% Plugging in the estimation in \Cref{lemma_Lyapunov_alpha_PL} to \Cref{Thm: lyapunov method}, we have the following corollary.
% \begin{corollary} \label{corollary_reduce_PI_gibbs_to_PI_truncated_gibbs}
%     Suppose that the assumptions in \Cref{lemma_Lyapunov_alpha_PL} hold and further suppose that $\epsilon$ is sufficiently small so that $\sigma = \frac{\errorboundconstant^2}{8\epsilon^2} \bigg( C\epsilon\bigg)^{\frac{2}{\alpha}}$ in \cref{eqn_estimation_sigma_b}.
%     For any subdomain $U$ such that $\optimalset{(C\epsilon)^{\frac{\alpha-1}{\alpha}}} \subseteq U \subseteq \optimalset{c(C\epsilon)^{\frac{\alpha-1}{\alpha}}}$, for some constant $c>1$, we have
%     \begin{equation}
%         \PIconstant_\GibbsMeasure \geq \frac{\frac{\errorboundconstant^2}{8\epsilon^2} \bigg( C\epsilon\bigg)^{\frac{2}{\alpha}}}{\frac{\errorboundconstant^2}{4\epsilon^2} \bigg( C\epsilon\bigg)^{\frac{2}{\alpha}}\cdot c^{\frac{2-\alpha}{\alpha-1}} + \PIconstant_{\epsilon, U}} \PIconstant_{\epsilon, U},
%     \end{equation}
%     where we recall $\PIconstant_\GibbsMeasure$ denotes the \Poincare\ constant for the Gibbs measure $\GibbsMeasure$ while $\PIconstant_{\epsilon, U}$ denotes the one for the truncated Gibbs measure $\TruncatedGibbsMeasure$.
% \end{corollary}
% \red{Consequently, if $\PIconstant_{\epsilon, U}$ is bounded from above, $\PIconstant_\GibbsMeasure \geq \frac{\lim_{\epsilon\rightarrow 0} \PIconstant_{\epsilon, U}}{c^{\frac{2-\alpha}{\alpha-1}}}$ as $\epsilon \rightarrow 0$.}

\subsection{Perturbation near the optimal set} \label{section_perturbation_near_optimal_set}
Applying the estimates from \Cref{lemma_Lyapunov_alpha_PL} to \Cref{Thm: lyapunov method}, we have reduced the estimation for the \Poincare\ constant of the Gibbs measure $\GibbsMeasure$ to the estimation for the one of the truncated Gibbs measure $\TruncatedGibbsMeasure$.
Unfortunately, the latter remains elusive.
In this section, we further reduce the estimation of $\PIconstant_{\epsilon, U}$ to the estimation of the Neumman eigenvalue of the Laplacian operator on $U$.

Since all the points in the subdomain $U$ are sufficiently close to the optimal set $\optimalset{}$, one can utilize the Taylor expansion of the potential $V$ to show that the density function of the truncated Gibbs measure $\TruncatedGibbsMeasure$ is an $\epsilon$-perturbation of the uniform density function on $U$. We can utilize the perturbation principle in \Cref{theorem_perturbation_principle} and the smoothnesss of $V$ (see the first point in \Cref{lemma_useful_constants}) to derive the following result.

\begin{my_lemma} \label{lemma_perturbation_of_truncated_Gibbs}
% [Perturbation of truncated Gibbs measure by an $\epsilon$-modification]
Suppose that the assumptions and requirements in \Cref{lemma_Lyapunov_alpha_PL} hold.
% Suppose that \Cref{ass_growth_V,ass_optimalset_interior,ass_PL,ass_regularity_V} hold.
% Suppose the temperature $\epsilon$ is sufficiently small such that
        % $\epsilon \leq \frac{\errorboundconstant^2}{64C_g} \text{ and } \left(C\epsilon\right)^{\frac{\alpha-1}{2}} \leq \delta_V$.
% For any subdomain $U\subseteq \optimalset{c(C\epsilon)^{\frac{\alpha-1}{\alpha}}}$ for some constant $c>1$,
If the uniform measure $\LebesgueMeasure{U}$ satisfies PI($\rho_{U}$), then the truncated Gibbs measure $\TruncatedGibbsMeasure$ also satisfies PI($\rho_{\epsilon, U}$) with
\begin{equation*}
\setlength{\abovedisplayskip}{5pt}
            \setlength{\belowdisplayskip}{5pt}
\exp\{\bar C\}\rho_{\epsilon, U} \geq  \rho_{U} = \NeummanEigenvalue(U),
\end{equation*}
where $\bar C = 4L C$ and $\NeummanEigenvalue(U)$ is the Neumann eigenvalue on $U$.
\end{my_lemma}
Combining \Cref{lemma_Lyapunov_alpha_PL,lemma_perturbation_of_truncated_Gibbs}, we obtain the main conclusion of this section.

\begin{theorem}\label{corollary_reduce_PI_to_Neumann_eigenvalue}
    Suppose that the requirements in \Cref{lemma_Lyapunov_alpha_PL} hold and further suppose that $\epsilon$ is sufficiently small so that $\sigma = \frac{d\mu^{-}}{2\epsilon}$ in \cref{eqn_estimation_sigma_b}.
    By choosing $U = \optimalset{\sqrt{C\epsilon}}$, we have 
        $\PIconstant_\GibbsMeasure \geq \frac{1}{2}\frac{d\mu^{-}}{d\mu^{-}+M_\Delta}\exp(-\bar C) \NeummanEigenvalue(U)$.
    % where we recall $\PIconstant_\GibbsMeasure$ denotes the \Poincare\ constant for the Gibbs measure $\GibbsMeasure$.
    % while $\PIconstant_{U}$ denotes the one for the Lebesgue measure on $U$.
\end{theorem}

% \red{Connect PI of Lebesgue measure with Neumman eigenvalue. 
% Want to use Neumman eigenvalue on $\optimalset{}$ to bound.
% Turn into a stability analysis problem.
% Highlight the difficulty: Lipschitz boundary.
% }

% \clearpage
\section{Step 2: Stability Analysis of the Neumman Eigenvalue} \label{section_stability_neumann_eigenvalue}
In the previous section, we reduce the estimation of the \Poincare\ constant of $\GibbsMeasure$ to $\NeummanEigenvalue(U)$, the Neumann eigenvalue of the Laplacian operator on the subdomain $U$.
In this section, we justify the choice of $U = \optimalset{\sqrt{C\epsilon}}$. Recall that $\optimalset{}$ is a $\mathcal{C}^2$ embedding submanifold. Hence, $U$ matches the tubular neighborhood of $\optimalset{}$, a special kind of  \emph{neighborhood in the ambient space}, as described in \Cref{Thm:tubular neighborhood}. For the simplicity of notation, we denote $\tilde \epsilon = \sqrt{C\epsilon}$ in this section.
% We take $U = T(\tilde \epsilon)$ with $\tilde \epsilon = \sqrt{C\epsilon}$.
% Moreover, the tubular neighborhood coincides with the expansion of $\optimalset{}$, i.e. $T(\tilde \epsilon) = \optimalset{}^{\tilde \epsilon}$.

\subsection{Tubular Neighborhood of a $\mathcal{C}^2$ Embedding Submanifold}
As an important property of embedding submanifold in $\R^d$, let us introduce the tubular neighborhood theorem (see \cite{DiffTopo1974}, Page 69). 
\begin{theorem}[Tubular neighborhood theorem]\label{Thm:tubular neighborhood}
Let $M$ be a compact $\mathcal{C}^2$ embedding submanifold in $\R^d$. 
Let $T(\tilde \epsilon) = M^{\tilde \epsilon} = \{ y \in \R^d : |y - m| \leq \tilde \epsilon, m \in M\} $ be the tubular neighborhood of $M$.
% , defined as 
% \begin{equation*}
%     T(\tilde \epsilon) = \{ y \in \R^d : |y - m| \leq \tilde \epsilon, m \in M\}.
% \end{equation*}
There exists a positive number $\tilde \epsilon_\mathrm{TN}$, such that for all $0 < \tilde \epsilon \leq \tilde \epsilon_\mathrm{TN}$, one has
% \begin{enumerate}
    (1) each $y \in T(\tilde \epsilon)$ possesses a unique closest point $\pi(y) \in M$;
    (2) the projection map  $\pi: T(\tilde \epsilon) \rightarrow M$ is a submersion. That is to say, the linear map $D_y \pi: T_y T(\tilde \epsilon) \rightarrow T_{\pi(y)}M$ is surjective at each point $y \in T(\tilde \epsilon)$.
% \end{enumerate}
% Moreover, if $M$ is compact, then $\epsilon(x)$ can be taken to be a constant.
\end{theorem}

% \begin{my_remark}
\noindent There is a more concrete representation for the tubular neighborhood, which we adopt in the rest of the paper: For any $y \in T(\tilde \epsilon)$, $y$ can be written as $y = m + \nu$, where $m$ is a point on $M$ and $\nu \perp M$ at $m$ with $|\nu| \leq \tilde \epsilon$, and the map $y \rightarrow (m, \nu)$ is a diffeomorphism. More precisely, under the local chart $(\Gamma, \phi)$, the diffeomorphism $y \rightarrow (m, \nu)$ can be written as
\begin{equation}\label{local coordinate decomposition}
\setlength{\abovedisplayskip}{5pt}
            \setlength{\belowdisplayskip}{5pt}
y(u, r) = \mathcal{M}(u) + \sum_{l=k+1}^dr^l \N_l(u),\ \ \ \ u \in \Gamma \subset \R^k, \ \ \ (r^{k+1},...,r^d) \in B(\tilde \epsilon) \subset \R^{d-k}.
\end{equation}
Here $\M(u)$ is local coordinate representation of the including map $i_M: M \rightarrow \R^d$ in \eqref{embedding structure}, $\N_{k+1},...,\N_d: \Gamma \rightarrow \R^d$ are $d-k$ normal vector fields on $M$ which are also orthogonal to each other,
and $B(\tilde \epsilon)$ denotes ball with radius $\tilde \epsilon$ in $\R^{d-k}$. We refer readers to \Cref{Rie structure on local chart} for more details about these vector fields on the local chart $(\Gamma, \phi)$. For brevity, we also denote $r = (r^{k+1},...,r^d)$ with $|r| \leq \tilde \epsilon$. 
% Readers can find more details about this version in \citep[Theorem 6.24]{Lee2020IntroductionTS}. \red{refer to appendix}
% \end{my_remark}

\subsection{Stability of the Neumann Eigenvalue on the Tubular Neighborhood}
We now focus on the stability of $\NeummanEigenvalue(T(\tilde \epsilon))$ w.r.t. $\tilde \epsilon$, where $T(\tilde \epsilon)$ is a tubular neighborhood of $\optimalset{}$.
Given the special structure of $T(\tilde \epsilon)$, our idea is to exploit the tensorization of the \Poincare\ inequality. 

\begin{my_proposition}[Proposition 4.3.1 in \citep{Bakry2013AnalysisAG}] 
\label{theorem_tensorization}
Let $(E_1, \mu_1)$ and $(E_2, \mu_2)$ be two probability spaces with measure $\mu_1$ and $\mu_2$, and they satisfy PI with constants $C_1$ and $C_2$ respectively. Then the product space $(E_1 \times E_2, \mu_1 \times \mu_2)$ satisfies a PI with constant $C = \max\{C_1, C_2\}$.
\end{my_proposition}

To utilize the above theorem, recall the min-max variational principle of PI in \Cref{section_eigenvalue}, which consists of the $L^2$ norm and the $\sobolevspace{1,2}$ norm on $T(\tilde \epsilon)$.
To bound these integrals, we show that the uniform measure on $T(\tilde \epsilon)$ can be decomposed as the product of a pair of decoupled measures on the manifold $\optimalset{}$ and the subspace of the normal coordinates $r$, up to an $\mathcal{O}(\tilde \epsilon)$ perturbation:

% \red{The stability of the spectral is determined only by the normal direction; We actually exploit the spectral stability along $d-k$ normal directions of $\optimalset{}$.} 

\begin{itemize}[leftmargin=*]
    \setlength\itemsep{0em}
    \item We decompose the integral in $T(\tilde \epsilon)$ as the integral in the product space $\optimalset{} \times B(\tilde \epsilon)$ with an additional factor of order $1 + \mathcal{O}(\tilde \epsilon)$. This is possible since for any $x \in \optimalset{}$, an $\epsilon_1$-neighborhood of $x$ under the topology of $\R^d$ can be viewed, up to a diffeomorphism, as the product space $B^{\optimalset{}}(x, \epsilon_2) \times B(\epsilon_3)$. Here $B^{\optimalset{}}(x, \epsilon_2) \subset {\optimalset{}}$ is a ball in $\optimalset{}$ with center $x$ and radius $\epsilon_2$ (defined according to the geodesic distance on $\optimalset{}$) and $B(\epsilon_3)$ is a ball in $d - k$ normal directions of $\optimalset{}$ at $x$ with radius $\epsilon_3$.
    Equivalently, we turn the uniform measure on $T(\epsilon)$ to the product of the volume measure induced by the including map $i_\optimalset{}$ on $\optimalset{}$ and the uniform measure on $B(\tilde \epsilon)$.

    % For any given $x \in \optimalset{}$, the $\epsilon$ neighborhood of $x$ in tubular neighborhood can be viewed, up to a diffeomorphism, as the product space $B^{\optimalset{}}(x, \epsilon') \times B(\epsilon)$. Here $B^{\optimalset{}}(x, \epsilon') \subset {\optimalset{}}$ is a ball in $\optimalset{}$ with center $x$ and radius $\epsilon'$, and $B(\epsilon)$ is a ball in $d - k$ normal directions of $\optimalset{}$ with center $x$ and radius $\epsilon$.
    
    \item With the above decomposition, we show that both $L^2$ and $\sobolevspace{1,2}$ norm on $T(\tilde \epsilon)$ are bounded by their counterparts on the the product space $\optimalset{} \times B(\epsilon)$ with an $\tilde \epsilon$ perturbation.

    % Using the boundedness condition of the second fundamental form of $\optimalset{}$, we can prove that the $L^2$ norm and the $\sobolevspace{1,2}$ norm in  $T(\epsilon)$ are equivalent with  the $L^2$ norm and the $\sobolevspace{1,2}$ norm in the product space $\optimalset{} \times B(\epsilon)$ with small $\epsilon$ perturbation.  \red{explain equivalent: mutually control?}

    % \item By exploiting its min-max variational formulation, we lower bound $\NeummanEigenvalue(U)$ by the Neumann eigenvalue of the Laplacian operator on $\optimalset{} \times B(\epsilon)$ up to an $\mathcal{O}(\epsilon)$ perturbation. We can easily compute the latter by the tensorization property of Poincar\'e inequality. 

\end{itemize}
\begin{figure}
    \captionsetup{format=plain}
    \centering
	\begin{tabular}{c @{\quad } c @{} c}
		\includegraphics[height=.2\columnwidth]{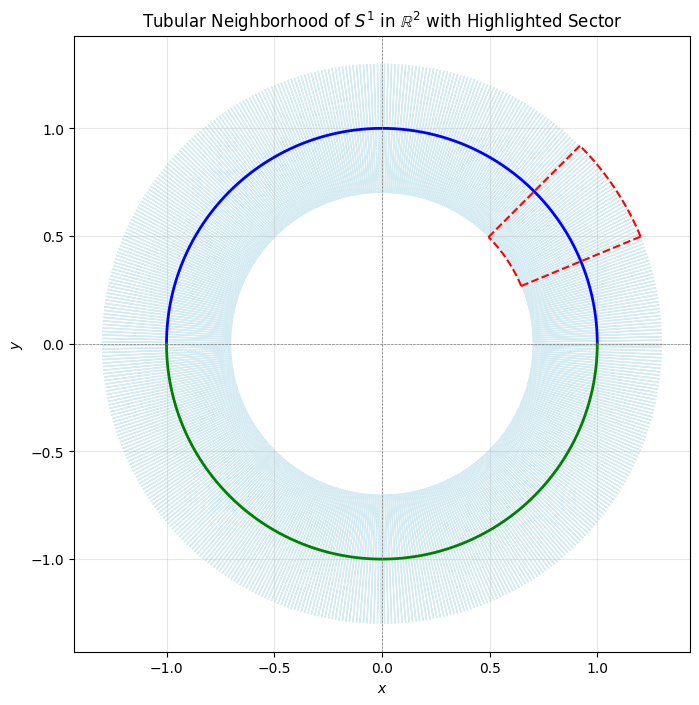} &
		\includegraphics[height=.2\columnwidth]{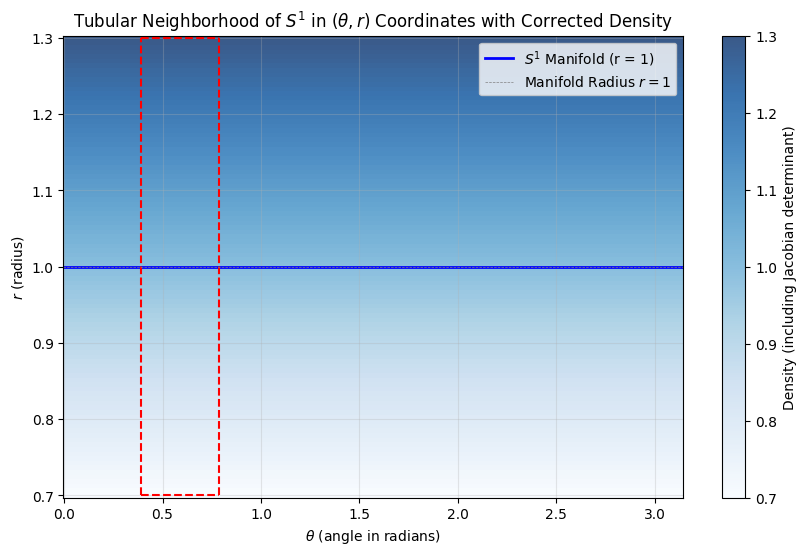} &
        \includegraphics[height=.2\columnwidth]{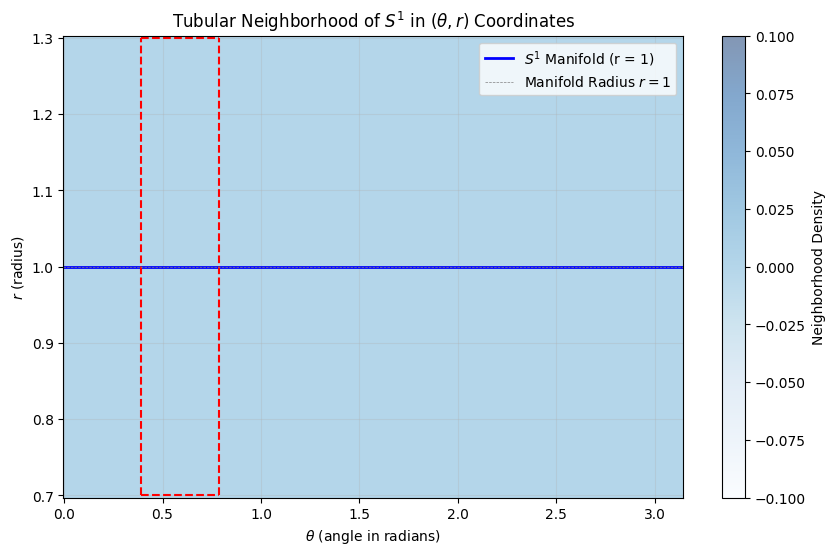} \\
        (a) & (b) & (c)
	\end{tabular}
    \caption{
    The circle in (a) can be represented using two local charts (blue and green). 
    Using the tubular neighborhood theorem, in a local region of $U$ (outlined with the red dashed line)
    we transform the uniform measure to a pair of decoupled measures on the tangent and normal directions. (a) Uniform measure $\mu_U$ (over $(x, y)$) under the Cartisian coordinate $(x, y)$; (b) Uniform measure $\mu_U$ (over $(x, y)$) under the local coordinate $(\theta, r)$; (c) Uniform measure (over $(\theta, r)$) under the local coordinate $(\theta, r)$. Importantly, when the radius of the tubular neighborhood is small, the densities in (b) and (c) point-wisely control each other.
    % The above derivation already allows us to exploit the min-max formulation of the Laplacian-Beltrami eigenvalue and the tensorization of the \Poincare\ inequality to conclude the target result. 
    % Note that the above transformation only holds locally. To reach the global conclusion, we will use the standard technique of partition of unity \citep[Chapter 13]{tu2010introduction}.
    }
\end{figure}

{With the above derivation, we turn our focus to the PIs of the two decoupled measures. The PI constant of the volume measure on $\optimalset{}$ is inherent to $\optimalset{}$ and is \emph{temperature-independent}. The PI constant of the measure on the normal coordinates has been explicitly calculated in the literature.
}
\begin{my_proposition}(PI for Lebesgue measure on a ball, \cite[Page 293, Theorem 2]{evans10}) 
Let $B(\tilde \epsilon) \subseteq \R^{d-k}$ be a ball with radius $\tilde \epsilon$.
Let $\mu_{B(\tilde \epsilon)}$ be the uniform measure over $B(\tilde \epsilon)$. There exists a constant $\tilde C$ depending only on the dimension $(d-k)$ such that the PI constant $\PIconstant_{\mu_{B(\tilde \epsilon)}} \geq \frac{1}{\tilde C\tilde \epsilon}$.

% Assume $1 \leq p \leq \infty$. Then there exists a constant $C$, depending only on $d$ and $p$, such that
% \begin{equation*}
% \left\{\int_{B(\tilde \epsilon)} \left(f - \int_{B(\tilde \epsilon)} f \ud x \right)^p \ud x \right\}^{\frac{1}{p}} \leq C\tilde \epsilon \left\{ \int_{{B(\tilde \epsilon)}} |\nabla f|^p \ud x \right\}^{\frac{1}{p}},
% \end{equation*}
% for each ball $B(\tilde \epsilon) \subset \R^d$ and each function $f \in W^{1,p}(B(\tilde \epsilon))$.
\end{my_proposition}

Combining with \Cref{theorem_tensorization}, we know $\NeummanEigenvalue(B(\tilde \epsilon)$ is dominated by $\Eigenvalue(\optimalset{})$ when $\tilde \epsilon$ is sufficiently small, and hence, $\NeummanEigenvalue(T(\tilde \epsilon)$ is determined by $\Eigenvalue(\optimalset{})$.
We now make the above reasoning rigorous.

% A key step to make the above reasoning rigorous is the following change of variables formula.
\begin{my_lemma}\label{lemma_wely's formula}\citep{Weyl1939OnTV} Let $\varphi: T(\tilde \epsilon) \rightarrow \R$ be an integrable function, then we have 
\begin{equation*}
\setlength{\abovedisplayskip}{5pt}
            \setlength{\belowdisplayskip}{5pt}
\int_{T(\tilde \epsilon)} \varphi(y) dy = \int_{\optimalset{}} \Big\{ \int_{B(\epsilon)} \varphi(y(u,r)) \Big| \text{det} \big(I_k + \sum_{l = k+1}^d r^l \tilde{G}(l) \big) \Big| dr^{k+1}...dr^d \Big\} d\M,
\end{equation*}
where the variable $y$ on the right hand side uses the expression of the local coordinate \eqref{local coordinate decomposition}, $I_k$ is $k \times k$ identity matrix and 
$\tilde{G}(l), l = k+1,...,d$ are defined in the end of \Cref{subsection: submanifold}.
% $\tilde{G}(l)$ is defined in \eqref{matrix tilde G} in the appendix.
\end{my_lemma}

To relate $\NeummanEigenvalue(U)$ to $\Eigenvalue(\optimalset{})$ through their min-max variational principles, we need the following expression of the gradient under change of variables.
\begin{my_lemma}\label{lemma_transformation}
Let $\varphi(y) \in W^{1,2}(T(\tilde \epsilon))$, then, on each local chart $(\Gamma, \phi)$, we have in the weak sense
\begin{equation*}
\setlength{\abovedisplayskip}{5pt}
            \setlength{\belowdisplayskip}{5pt}
\nabla_{(u,r)} \varphi(y(u,r)) = \nabla_y \varphi \cdot \left( [\frac{\partial \M}{\partial u^1},...,\frac{\partial \M}{\partial u^k}, \N_{k+1},...,\N_d] 
\left[ 
\begin{matrix} 
I_k + \sum_{l=k+1}^d r_l \tilde{G}(l) & 0 \\
0 & I_{d-k} 
\end{matrix}
\right] \right).
\end{equation*}

% where $\tilde{G}(l), l = k+1,...,d$ are defined in the end of \Cref{subsection: submanifold} and $I_k$ is $k \times k$ identity matrix. 
\end{my_lemma}

We are now ready to state the estimation of $\NeummanEigenvalue(U)$. 
Note that the choice of $U = T(\epsilon)$ can also be regarded as a domain expansion, so the following result is a stability analysis.
However, one should keep in mind that the domain expansion is performed under the topology of the ambient space $\R^d$, so $U$ and $\optimalset{}$ do not have the same dimension.

\begin{my_proposition}\label{proposition_stability}
Suppose that \Cref{ass_growth_V,ass_no_saddle,ass_optimalset_interior,ass_PL} hold, then we have the non-asymptotic estimates of $\NeummanEigenvalue(T(\tilde \epsilon))$ based on $\Eigenvalue(\optimalset{})$ 
\begin{equation*}
\setlength{\abovedisplayskip}{5pt}
            \setlength{\belowdisplayskip}{5pt}
\Eigenvalue(\optimalset{})(1 - B \tilde \epsilon) \leq \NeummanEigenvalue(T(\tilde \epsilon)) \leq \Eigenvalue(\optimalset{})(1 + B \tilde \epsilon),
\end{equation*}
for some constant $B = B(d, k, \tilde{G}(l))>0$ when $\tilde \epsilon$ is small enough. 
% Here $k$ denotes the dimension of $\optimalset{}$ and $\tilde{G}(l)$ denotes the second fundamental form of $\optimalset{}$ defined in \Cref{Rie structure on local chart}.
\end{my_proposition}

\section{\Poincare\ Inequality for the \logPLmeasure\ Measure}\label{section of main result}
We now combine the results in \Cref{section_reduction_to_neumann_eigenvalue,section_stability_neumann_eigenvalue} to conclude the \Poincare\ inequality for the Gibbs measure $\GibbsMeasure$. Please find the proof in \Cref{proof_thm_main}.
\begin{theorem}\label{thm_main}
    Suppose that the requirements in \Cref{lemma_Lyapunov_alpha_PL} hold.
    % We have the following conclusions for the two cases in \Cref{ass_optimalset_interior}.
    % \begin{enumerate}
    % \item[$(\CIRCLE)$] (Lipschitz Domain of $\R^d$) Pick $c$ and $A$ according to \Cref{corollary_neumann_eigenvalue_U} and suppose that the temperature $\epsilon$ in addition satisfies $\epsilon \leq \min \big\{ \frac{\delta}{4}, \frac{1}{C} \Big( \frac{1}{2Ac}\Big)^{\frac{\alpha}{\alpha - 1}} \big\}$, where $\delta$ is defined in \Cref{definition_lipschitz_boundary}. 
    % Recall that $\NeummanEigenvalue(\optimalset{})$ denotes the Neumann eigenvalue on $\optimalset{}$ define in \Cref{definition_Neumann_eigenvalue}.
    % We have
    % % For the case $(\CIRCLE)$ in which $\optimalset{}$ is a Lipschitz Domain of $\R^d$, we have 
    % \begin{equation*}
    %     \PIconstant_\GibbsMeasure \geq C_{P} \NeummanEigenvalue(\optimalsetinterior{}).
    % \end{equation*}
    % \item[$(\Circle)$] (Embedding submanifold of $\R^d$) 
    Suppose that the temperature $\epsilon$ in addition satisfies $\epsilon \leq \min \big\{\tilde \epsilon_\mathrm{TN}, \frac{1}{C} \Big( \frac{1}{2B}\Big)^{2} \big\}$, where $\tilde \epsilon_\mathrm{TN}$ and $B$ appear in \Cref{Thm:tubular neighborhood} and \Cref{proposition_stability} respectively.
    Recall that $\Eigenvalue(\optimalset{})$ denotes the eigenvalue of the Laplacian-Beltrami operator  (\Cref{def_eigenvalue_laplacian_beltrami}).
    Define a constant $C_P = \frac{1}{4}\frac{d\mu^{-}}{d\mu^{-}+M_\Delta}\exp(-4LC)$.
    We have 
    % \begin{equation*}
        \(\PIconstant_\GibbsMeasure \geq C_{P} \Eigenvalue(\optimalsetinterior{})\).
    % \end{equation*}
    % \end{enumerate}
    % Here $C_p = \frac{1}{4}\frac{d\mu^{-}}{d\mu^{-}+M_\Delta}\exp(-\bar C)$ is a temperature-independent constant.
\end{theorem}
By noting that $\Eigenvalue(\optimalsetinterior{})$ is an inherent property of the optimal set $\optimalset{}$ and hence the potential function $V$, we reach the target conclusion, i.e. $\PIconstant_\GibbsMeasure = \Omega(1)$.
With the connection between the convergence of Langevin dynamics and the \Poincare\ inequality, we have that for a \PLcirc\ potential, the Langevin dynamics converges at the rate $\tilde{\mathcal{O}}(\frac{1}{\epsilon})$.
The convergence of discrete-time algorithms (in the sense of R\'enyi-divergence) like LMC can be readily derived by combining our result with \citep[Theorem 7]{chewi2024analysis}.
To the best of our knowledge, this is the first work to study the PI constant for potentials with non-isolated minima in the low temperature region. Interestingly, for the non-log-concave $\logPLmeasure$ measures, we still yield an $\Omega(1)$ lower bound for the PI constant.

\paragraph{Conclusion and future work}
We study the \Poincare\ constant of the \logPLmeasure\ measures as a template to understand the convergence behavior of stochastic dynamics on potentials with non-isolated minima. We relate the corresponding \Poincare\ constant to the spectral property of the Laplacian-Beltrami operator on the optimal set $\optimalset{}$, and establish a \emph{temperature-independent} lower bound.
Our next steps are (1) improving PI to LSI, (2) relaxing the $\mathcal{C}^2$ manifold to the $\mathcal{C}^1$ case, and (3) studying the non-isotropic noise case.

\bibliography{sample}
\clearpage
\appendix
\section{Related Work} \label{section_related_work}
% \red{Informal statements of the results in \Cref{section_reduction_to_neumann_eigenvalue,section_stability_neumann_eigenvalue}.}
% We highlight again that \citet{AOP} assume the potential $V$ to be Morse, i.e. $\nabla^2 V(x)$ is invertible for all critical point $x$. This assumption implies that the minima of $V$ is collection of singletons.
% In contrast, we allow the Hessian $\nabla^2 V(x)$ to be degenerate on the optimal set $\optimalset{}$ and focus on the case where $\optimalset{}$ is non-singleton.
% \red{Stability of Neumann eigenvalue.}
% \clearpage
% \paragraph{Functional inequalities.} 
\Poincare\ inequality is a crucial topic in domains like probability, analysis, geometry and so on. We categorize the results according to the property of the energy.
\begin{itemize}
    \item When the potential function $V$ is strongly convex, the famous Bakry-Emery criterion ensures that the PI constant is of order $\Theta(\frac{1}{\epsilon})$ \citep{bakry2006diffusions}. See for more detailed discussion in the book \citep{Bakry2013AnalysisAG}.
    \item 
    When the potential function $V$ is convex, there two prominent strategies to study the PI constant: the Lyapunov function approach \citep{PI_Lyapunov} and the approach initiated by \citet{cheeger1970lower} which relates the PI constant to the isoperimetric constant of the target measure and the KLS conjecture \citep{Kannan1995IsoperimetricPF,lee2024eldan}.\\
    % the spectral methods \citep{wu2000uniformly}, etc.. We focus on the first two strategies.\\
    Following \citep{bakry2008rate}, \citet{PI_Lyapunov} reduce the \Poincare\ inequality on $\R^d$ into a small compact region $U$ if $V$ satisfies a Lyapunov condition. Convex functions that are exponentially integrable is proved to satisfy this condition and hence the corresponding log-concave measure satisfies \Poincare\ inequality. However, they do not consider the low temperature region and if we naively utilize the Holley-Stroock perturbation principle to derive the \Poincare\ inequality in the said compact region $U$, the resulting \Poincare\ constant on $\R^d$ is $\Omega(\exp(\frac{1}{\epsilon}))$.\\    
    % one of most important method to derive Poincare inequality is lyapunov method.  The paper \cite{Cattiaux2016HittingTF} also gives more general connections of Poincare constant, lyapunov condition of $V$ and ergodicity in probability theory.\\    
    Another important research line is Cheeger's inequality \citep{cheeger1970lower}, which relates the \Poincare\ constant on a compact set with the Cheeger constant, describing the geometrical property of the set.
    Later on, the KLS conjecture extends the Cheeger constant in Eucildean space with log-concave measure in \citep{Kannan1995IsoperimetricPF}. We recommend a good survey \citep{lee2018kannanlovaszsimonovitsconjecture} and reference therein for more precise introduction and recent progress. We mention a very important method to prove this conjecture --- stochastic localization, which starts from Eldan's work in \citep{Eldan2013ThinSI}, and then generalized by \citet{lee2024eldan} and \citet{chen2020AnAC}. We note that KLS conjecture pays more attention to the independence of dimension.
    However, it also helps to derive the relationship of temperature and \Poincare\ constant for log-concave measure.
    In fact, if we combine the Lyapunov function approach with the KLS conjecture, we can prove that for a log-concave measure, the \Poincare\ constant remains temperature-independent in the low temperature regime.
    % See also a on-going book for a very nice overview on sampling from Log-concave measures \citep{sinho2024logconcave}.
    \item The convergence behavior of the Langevin dynamics on a non-log-concave measure is also a crucial research problem in various domains. In particular, here we focus on the case where the potential function $V$ has at least two separated (local) minima. As mentioned in the introduction, in this case, there is a two time-scale phenomenon in the convergence behavior in the low temperature region.
    The exponential dependence on the inverse temperature in the slow scale is classically known as the Arrhenius law \citep{arrhenius1967reaction}, which can be proved for example by the Freidlin-Wentzell theory on large deviation \citep{freidlin2012random}.
    With additional assumptions that $V$ is a Morse function and its saddle points has exactly one negative eigenvalue, this subexponential factor in the convergence behavior is captured in \citep{eyring1935activated,kramers1940brownian} and rigorously proved by \citep{bovier2004metastability,gayrard2005metastability} through potential theory.
    \citet{AOP} study the same problem, but through the functional inequality perspective. Following a two-scale approach, \citet{AOP} split the variance (the term $\VAR_\mu(f)$ in \cref{eqn_poincare_inequality}) into local variances and coarse-grained variances. The estimations on both variances can be combined together to obtain the global \Poincare\ inequality.
\end{itemize}

\section{\Poincare\ constant for log-concave measure in the low temperature regime} \label{section_pi_log_concave}
{Suppose} the potential function $V$ is convex and WLOG $V^*=0$. Suppose that $V$ is exponentially integrable for $\epsilon=1$, i.e. $\int_{\R^d} \exp(-V(x))\ud x < \infty$. We prove that the \Poincare\ constant $\PIconstant_\GibbsMeasure$ is $\Omega(1)$. Our proof is a combination of the Lyapunov function approach \citep{PI_Lyapunov} and the KLS conjecture \citep{lee2024eldan}.

Recall point (2) in \citep[Lemma 2.2]{PI_Lyapunov} which states that under the integrablity assumption, one has
\begin{equation*}
    \exists \alpha > 0, R>0, \mathrm{s.t.} \forall |x| \geq R, V(x) - V(0) \geq \alpha |x|.
\end{equation*}
Moreover, let us choose a Lyapunov function as $W(x) = \exp(\tilde W(\gamma |x|))$ where $\gamma = \frac{\alpha}{3}$
\begin{equation*}
    \tilde W(z) =
    \begin{cases}
        z &\quad z \geq R\\
        -\frac{12}{R^2} z^3 + \frac{28}{R} z^2 - 19z + 4R &\quad R/2\leq z \leq R\\
        0  &\quad z \leq R/2
    \end{cases}.
\end{equation*}
We can compute for $|x|\geq R$
\begin{equation*}
    \epsilon^{-1}\generator W(x) =  \gamma \left(\frac{n-1}{|x|} + \gamma - \frac{x\cdot \nabla V(x)}{\epsilon|x|}\right)W(x).
\end{equation*}
Note that $W\in \mathcal{C}^2$ and $W(x)\geq 1$ for all $x\in\R^d$.
Recall \Cref{Thm: lyapunov method}.
Set $\theta = \frac{\alpha}{\epsilon} - \gamma - \frac{(d-1)}{R}$ and set
\begin{equation*}
    b = \theta + \sup_{\|x\|\leq R} \epsilon^{-1} \generator W(x).
\end{equation*}
With with parameters $\theta, b, U=B(0, R)$, $W$ is a valid Lyapunov function. Moreover, it can be easily checked that $b \leq \theta + \frac{C_1}{\epsilon} + C_2$ for some temperature-independent constants $C_1$ and $C_2$ since both $V$ and $W$ are $\mathcal{C}^2$ and $W$ is a constant for $\|x\|\leq R/2$.

Finally, the \Poincare\ constant for the truncated Gibbs measure can be bounded using the KLS conjecture:
According to \citep[Theorem 13]{lee2024eldan}, the Cheeger constant of a log-concave measure can be lower bounded by (a polynomial of) the spectral norm of its covariance matrix. Since the truncated Gibbs measure is supported on a compact set, its Cheeger constant is lower bounded by a temperature-independent constant. Using the Cheeger's inequality, we have the conclusion.

\section{Properties of \PLcirc\ functions and \logPLmeasure\ measures}
\subsection{\logPLmeasure\ measures has a single modal} \label{section_log_PL_uni_modal}

\begin{proof}[\Cref{lemma_unimodal}]
    Note that by \Cref{ass_coercivity}, $V$ is not a constant function.
    Recall \Cref{ass_PL}. We have that for any $\optimalset{}'$ and any $x\in \mathcal{N}(\optimalset{}')$, if $\nabla V(x) = 0$, $V(x) = \min_{x\in\mathcal{N}(\optimalset{}')} V(x)$, i.e. $x\in\mathcal{N}(\optimalset{}')$ is a local minimizer if it is critical point. 
    
    Let $X$ denote the set of local maxima of $V$, which implies that $\forall x'\in X, \nabla V(x') = 0$. 

    Some useful facts are listed as follows
    \begin{itemize}
        \item $X \cap \mathcal{N}(\optimalset{}') = \emptyset$  for any $\optimalset{}'$: Let $x' \in X \cap \mathcal{N}(\optimalset{}')$. One has $\nabla V(x') = 0$, i.e. $x'$ is a local minimizer in $\mathcal{N}(\optimalset{}')$. But this contradicts with the assumption that $x'$ is a local maximizer.
        \item $X$ is either empty or a collection of singletons: Since for any $x'\in X$ (if $X \neq \emptyset$), one has $\nabla V(x') = 0$. Therefore, by \Cref{ass_no_saddle}, $\nabla^2 V(x') \prec 0$. Hence every $x'$ is a strict local maximizer, i.e. $X$ is a collection of singletons.
        % \item There are finitely many distinctive local minima values.
        \item The set of all local minima of $V$ has at most a finite number of separated components: 
        We prove via contradiction.
        Suppose that the set of all local minima of $V$ has infinitely many separated components.
        From \Cref{ass_coercivity}, all local minima are contained within a compact set. Denote the collection of all separated components of local minima of $V$ as $\{\optimalset{}'_i\}$. For every $i$, pick a representative point $x_i \in \optimalset{}'_i$. Clearly all $x_i$ are contained in a compact set. From Bolzano–Weierstrass theorem, we have that $\{x_i\}$ (as an infinite sequence) has at least one accumulation point $x$. Since $V$ is $\mathcal{C}^1$, $\nabla V(x) = 0$, i.e. $x$ is also a critical point. 
        \begin{itemize}
            \item If $x\notin \mathcal{N}(\optimalset{})$ (recall the definition of $\mathcal{N}(\optimalset{})$ in \Cref{ass_no_saddle}), $x$ is a local maximum by \Cref{ass_no_saddle}, which contradicts with the fact that in every neighborhood of $x$ there is a local minimum (since $x$ is a accumulation point of $\{x_i\}$).
            \item If $x \in \mathcal{N}(\optimalset{})$, there exists some $\optimalset{}_{i'}$ such that $x \in \mathcal{N}(\optimalset{}_{i'})$. If $x \in \optimalset{}_{i'}$, it cannot be an accumulation point of $\{x_i\}$, as in $\mathcal{N}(\optimalset{}_{i'})$ there is only one representative point $x_{i'}$. If $x \in \mathcal{N}(\optimalset{}_{i'}) \backslash \optimalset{}_{i'}$, $\nabla V(x) \neq 0$, which again contradicts with the fact that $\nabla V(x) = 0$.
        \end{itemize}
    \end{itemize}
    We can now prove that \logPLmeasure is uni-modal via contradiction. Suppose that the local minima of $V$ has at least two separated components. Pick any two separated components $\optimalset{}'_1$ and $\optimalset{}'_2$ and let $x_1$ and $x_2$ be two points in these two components respectively. WLOG, assume $f(x_1) \geq f(x_2)$.
    % Now set $P = \mathcal{N}(\optimalset{}'_1) \cup \mathcal{N}(\optimalset{}'_2)$. Recall the following result from \citet{katriel1994mountain}.
    \begin{theorem}[\citet{katriel1994mountain}, Theorem 2.1]
    \label{thm:katriel}
        Let $f: \R^n \rightarrow \R$ be $\mathcal{C}^1$ and coercive. Let $x_1, x_2 \in \R^n$ and let $P \subset \R^n$ separate $x_1$ and $x_2$ (that is, $x_1$ and $x_2$ lie in different components of $\R^n\setminus P$ ), and:
    $$
    \max \left\{f\left(x_1\right), f\left(x_2\right)\right\}<\inf _{x \in P} f(x)=p
    $$
    Then there exists a point $x_3$ which is a critical point of $f$, with: $f\left(x_3\right)>\max \left\{f\left(x_1\right), f\left(x_2\right)\right\}$.
    Moreover, $x_3$ is either a local minimum or a global mountain passing point\footnote{A point $x$ is called a global mountain passing point of $f$ if for every neighborhood $\mathcal{N}(x)$, the set $\{y: f(y) < f(x)\}\cap\mathcal{N}(x) $ is disconnected.}.
    \end{theorem}
    To apply the above theorem, take $P$ to be a subset of $\mathcal{N}(\optimalset{}'_1) \backslash \optimalset{}'_1$ such that $\forall y \in P$, $V(y) > V(x_1) \geq \max\{ V(x_1), V(x_2)\}$. This is always possible according to \Cref{ass_PL}.
    Consequently, according to the above theorem, we can find a critical point $x_3$ which is either a local minimum or a global mountain passing point. 
    Consider two cases:
    \begin{enumerate}
        \item $x_3$ is not a local minimum but a global mountain passing point. Since $x_3$ is a critical point, \Cref{ass_no_saddle} implies that $x_3$ is a strict local maximum. However, for $d\geq 2$, this is not possible, as in this case a strict local minimum point is not a global mountain passing point.
        \item $x_3$ is a local minimum. If this is the case, we now pick $x_1$ and $x_3$ and apply \Cref{thm:katriel}. It gives us a new local minimum $x_4$ (note that $x_4$ cannot be a global mountain passing point as discussed in the first case). Important, note that every time we have $V(x_{i+1}) > V(x_{i})$ and that every $x_i$ is a local minimum. We can only do this a finite number of times since the collection of local minima of $V$ has at most a finite number of separated components (i.e. there can only be a finite number of different values of $V$ on the collection of local minima of $V$). 
    \end{enumerate}
    Consequently, none of the above two cases is possible and the local minima of $V$ has only one connected component, i.e. $\GibbsMeasure$ is uni-modal.
\end{proof}
\subsection{\PLcirc\ functions admits an embedding submanifold as its global optimal set}
\begin{lemma}(Global embedding submanifold) Suppose that $V$ satisfies \Cref{ass_PL,ass_no_saddle,ass_coercivity}, then the optimal set $\optimalset{}$ is an embedding submanifold of $\R^d$. 
\end{lemma}

\begin{proof}
According to Assumption \ref{ass_PL}, we know that for each $x \in \optimalset{}$, $V(x)$ satisfies the $2$-PL condition around $x$ in $\R^d$ with constant $\PLconstant$. Then we immediately see that $\optimalset{}$ is a $C^2$ embedding submanifold locally around $x$ by \citep[Lemma 2.15]{rebjock2024fast}. Moreover, we also know that $rank(\nabla^2 V)$ is a constant by \citep[Corollary 2.13]{rebjock2024fast} since we assume that $\optimalset{}$ is compact in Assumption \ref{ass_coercivity}. Using \citep[Theorem 8.75]{Boumal_2023}, we obtain that $\optimalset{}$ is a global embedding submanifold with dimension $d - rank(\nabla^2 V)$. 
\end{proof}

\begin{lemma}(No boundary)
Suppose that $V$ satisfies \Cref{ass_PL,ass_no_saddle,ass_coercivity}, then the optimal set $\optimalset{}$ is an embedding submanifold of $\R^d$ without boundary.
\end{lemma}

\begin{proof}
Assume the optimal set $\optimalset{}$ is a $k$ dimensional embedding submanifold with boundary $\partial \optimalset{}$. We immediately know that $\partial \optimalset{}$ is a $k-1$ dimensional submanifold around $\bar{x}$ by Theorem 5.11 in \citep{Lee2020IntroductionTS}. Then we take $\bar{x} \in \partial S$ and $\{x_n\}_{n=1}^{+\infty} \subset \text{int} \optimalset{}$ around $\bar{x}$ such that $x_n \rightarrow \bar{x}$ as $n \rightarrow +\infty$. Moreover, we can take $\bm{n}$ as a normal direction of $\partial \optimalset{}$ at $\bar{x}$ but $\bm{n} \bm \notin N_{\bar{x}}\optimalset{}$ such that
\begin{equation*}
x(r) = \text{exp}(r \bm{n}), \ \ x(0) = \bar{x},  \ \ x(r_n) = x_n \ \ \ r, r_n \in [0, \epsilon] 
\end{equation*}
with $r < \text{inj}(\bar{x})$ small enough. We denote that
\begin{equation*}
\bm{n}(r) = \frac{\partial}{\partial r} \exp(r \bm{n}) = x_{\ast}(\bm{n})\Big|_r \in T_{x(r)} \optimalset{}.
\end{equation*} 
By equivalence of $2$-PL condition in Assumption \ref{ass_PL} and Morse-Bott property of interior point of $\optimalset{}$ in \citep{rebjock2024fast} 
when $V(x) \in C^2$, we have
\begin{equation*}
\nabla^2 V(x_n)[\bm{n}(r_n)] = 0,
\end{equation*}
since $\bm{n}(r_n) \in Ker \nabla^2 V(x_n)$. Again by continuity of $\nabla^2 V$, we have
\begin{equation}\label{Contradiction}
\nabla^2 V(\bar{x})[\bm{n}] = \lim_{n \rightarrow +\infty} \nabla^2 V(x_n)[n(r)] = 0. 
\end{equation}
On the other hand, we already know that $\bm{n}$ is normal to $\partial \optimalset{}$ at $\bar{x}$ and $x(r) \in \optimalset{} \backslash \partial \optimalset{}$ when we take $r > 0$, hence $-\bm{n}$\footnote{
Let us clarify the definition of ``tangent space" and ``normal space" of a point which is at the boundary of a manifold. Taking $M$ to be a $k$ dimensional manifold with boundary $\partial M$, and $W$ to be an open set of $H^k = \{ (x_1,...,x_k) \in \R^k | x_k \geq 0 \}$, then we actually have $\partial W = W \cap \partial H^k$. Now we keep the notation $\bar{x} \in \partial M \subset M$, then there exists a local chart $(W, \phi)$ around $\bar{x}$, i.e. $\phi(a) = \bar{x}, a \in W$ and $\phi: W \rightarrow M$ is an embedding. \\
When we consider $\bar{x} \in M$, we have $T_{\bar{x}}M = \R^k$. This conclusion can be verified according to the definition of tangent space by $\R$-algebra $C^{\infty}_{\bar{x}}(M)$ in \cite[Section 22.4 and Figure 22.5]{tu2010introduction}. If we also have $M$ is a embedding submanifold of $\R^d$, then the normal space $N_{\bar{x}}(M)$ is the complement space of $T_{\bar{x}}M$ in $\R^d$. \\
However, when we consider $\bar{x} \in \partial M$, we need define $\bar{\phi} = \phi|_{\partial W}: \partial W \rightarrow \partial M$, which is a local chart on $\R^{k-1}$. Then we can say that $T_{\bar{x}}(\partial M) = \bar{\phi}_{\ast}(T_a(\partial W)) =  \bar{\phi}_{\ast}(\R^{k-1})$, which is a $k-1$ dimensional subspace of $T_{\bar{x}}(M)$. The normal space $N_{\bar{x}}(\partial M)$ is the complement space of $T_{\bar{x}}(\partial M)$ in $T_{\bar{x}}(M)$. This definition is not contradicted to the conclusion that $\partial M$ is a
$k-1$ submanifold of $M$ without boundary (See \cite[Section 22.3]{tu2010introduction}).
}
is the outward normal direction to $\optimalset{}$ at $\bar{x}$ according to local orientation. Again by Quadratic Growth property in \citep[Proposition 2.2]{rebjock2024fast}, we have
\begin{equation}\label{QG property}
0 < C \text{dist}^2(x(r), \optimalset{}) \leq V(x(r)) - V(\bar{x}), \ \ \ \ r \in [-\epsilon, 0).  
\end{equation}
for some constant $C(\PLconstant) > 0$. However, if we consider the Tayor expansion on the right hand side,
\begin{equation*}
\begin{aligned}
V(x(r)) - V(x(0)) & = \nabla^2 V(x(r))(x(r) - x(0))^2 + o(\| x(r) - x(0) \|^2) \\ & = r^2 \langle \nabla^2 V(\bar{x})[-\bm{n}], [-\bm{n}] \rangle + o(r^2) \\ & = o(r^2),
\end{aligned}
\end{equation*}
which is contradictory to \eqref{QG property} since 
\begin{equation*}
r^2 \sim \text{dist}^2(x(r), \optimalset{}) \leq o(r^2), \ \ \ \ r \in [-\epsilon, 0) 
\end{equation*}
when we take $\epsilon$ small enough. Now we finish the proof. 
\end{proof}

\section{Local coordinate representation of geometric structures in \Cref{subsection: submanifold}}\label{Rie structure on local chart}

\textbf{The first fundamental form.} Using the local chart $(\Gamma, \phi)$ of $\optimalset{}$, the Riemannian metric $g_{\optimalset{}}$ can be written as
\begin{equation}\label{the first fdmt form}    
g_{\optimalset{}} = \sum_{i,j=1}^k g_{ij}(u)du^idu^j, \ \ \ \ u \in \Gamma \subset \R^k,
\end{equation}
where 
\begin{equation*}
g_{ij}(u) = \frac{\partial \M(u)}{\partial u^i} \cdot \frac{\partial \M(u)}{\partial u^j}, \ u \in \Gamma \subset \R^k.
\end{equation*}
Here $\{{\partial \M}/{\partial u^i}\}_{i=1}^k$ are actually $k$ tangent vector fields of $k$-dimensional embedding submanifold $\optimalset{}$, they generate the tangent plane on each point of $\optimalset{}$, i.e.
\begin{equation*}
T_m \optimalset{} = \text{Span}\langle \frac{\partial \M(u)}{\partial u^1},...,\frac{\partial \M(u)}{\partial u^k} \rangle, \quad m = (m_1(u),...,m_d(u)) \in \optimalset{},
\end{equation*}
and ``$\cdot$'' is standard inner product on $\R^d$. Also, these $k$ tangent vector fields decide $d-k$
normal vector fields on $\optimalset{}$ by following global equation group
\begin{equation}\label{constrain of N-1}
\frac{\partial \M(u)}{\partial u^i} \cdot \mathcal{N}_l(u) \equiv 0,\ \ \ i = 1,...,k, \ \ l = k+1,...,d, \ \ \ u \in \Gamma \subset \R^k,
\end{equation}
moreover, we can take $\mathcal{N}_{k+1},...,\mathcal{N}_d$ as standard normal vector fields by following global equation group,
\begin{equation}\label{constrain of N-2}
\N_i(u) \cdot \N_j(u) \equiv \delta_{ij}, \ \ \ i,j = k+1,...,d, \ \ \ u \in \Gamma \subset \R^k,
\end{equation}
they generate the normal bundle $N$ on $\optimalset{}$, i.e.
\begin{equation*}
N(m) = \text{Span} \langle \N_{k+1}(u),...,\N_d(u) \rangle, \ \ \ \ m = \M(u) \in M, \ \ \ u \in \Gamma \subset \R^k,
\end{equation*}
\textbf{The second fundamental form.} We define the second fundamental form  as a symmetric quadratic form on the local chart $(\Gamma, \phi)$,
\begin{equation}\label{the second fdmt form}
\Pi = - \sum_{l=k+1}^d \bigg\{ t^l \sum_{i,j=1}^k G_{ij}(l) du^idu^j \bigg\}, \ \ \ u \in \Gamma \subset \R^k, \ \ t \in B(\epsilon) \subset \R^{d-k},
\end{equation}
where the matrix $G(l) = (G_{ij}(l))$ is symmetric for each $l = k+1,...,d$. We also use the following notation
\begin{equation}\label{matrix tilde G}
G_j^i(l) = \sum_{s = 1}^k g^{is} G_{sj}(l), \ \ \ \ l = k+1,...,d.
\end{equation}
It is easy to see that the matrix $\tilde{G}(l) = (G_j^i(l))$ is also symmetric since the metric tensor $(g_{ij})$ is symmetric. Then, matrices $\{G(l)\}_{l=k+1}^d$ of the second fundamental form $\Pi$ can be locally written as
\begin{equation*}
G_{ij}(l)(u) = \frac{\partial \M(u)}{\partial u^i} \cdot \frac{\partial  \N_l(u)}{\partial u^j}, \ \ \ \ \ u \in \Gamma \subset \R^k.
\end{equation*}
Moreover, by constraint \eqref{constrain of N-2} of normal vector fields $\N_l,\ l = k+1,...,d$, we also have
\begin{equation*}
G_{ij}(l)(u) = \frac{\partial \M(u)}{\partial u^i} \cdot \frac{\partial  \N_l(u)}{\partial u^j} = - \frac{\partial^2 \M(u)}{\partial u^i \partial u^j} \cdot \N_l(u), \ \ \ \ \ u \in \Gamma \subset \R^k, 	
\end{equation*}
which implies that the matrix $G(l) = (G_{ij}(l))$ is naturally symmetric and $\Pi$ is a symmetric quadratic form. The geometric meaning of the second fundamental form of Riemannian submanifold $\optimalset{}$ by embedding structure \eqref{embedding structure} is the projection of the variation of normal vector fields along the tangent space of Riemannian submanifold $(\optimalset{}, g_{\optimalset{}})$ based on the ambient space $(\R^d, g_E)$.

\section{Proofs of results in \Cref{section_reduction_to_neumann_eigenvalue}}
\subsection{Proof of \Cref{lemma_property_of_V}}
% \begin{proof}
    \paragraph{Property 1.}  We prove via contradiction. Suppose that there is a sequence $x_i \in \partial \mathcal{N}(\optimalset{})$ such that 
    $\lim_{i\rightarrow\infty} \dist(x_i, \optimalset{}) = 0$. Note that $\partial \mathcal{N}(\optimalset{})$ is bounded. Hence, WLOG, we assume $x_i$ is convergent, since otherwise we can always take a convergent subsequence. Denote $x = \lim_{i\rightarrow\infty} x_i$. By construction, $x \in \partial \mathcal{N}(\optimalset{})$, boundary of an open neighborhood of $\optimalset{}$, but since $\dist(x, \optimalset{}) = 0$, $x\in \optimalset{}$, which leads to a contradiction.

    \paragraph{Property 2.} Note that under \Cref{ass_PL,ass_no_saddle,ass_growth_V}, $X$ contains at most finitely many singletons. This can be proved via contradiction: Suppose that there is an infinite sequence $x_i \in X$. Note that $X$ is bounded by \Cref{ass_growth_V}. Hence, WLOG, we assume $x_i$ is convergent, since otherwise we can always take a convergent subsequence. Denote $x = \lim_{i\rightarrow\infty} x_i$. From the construction of $\{x_i\}$, in any neighborhood of $x$, there is another local maximum point, which contracts with \Cref{ass_no_saddle} and $V \in \mathcal{C}^2$, since $x$ is a strict local maximum.

    \paragraph{Property 3.} We know from \Cref{ass_growth_V}, $\nabla V(x) \neq 0$ for any $x$ beyond a compact set. Denote this compact set by $\mathcal{Y}$. We now focus our discussion in $\mathcal{Y}$ and prove via contradiction.
    Suppose that within the said compact set there is a sequence $x_i \in \mathcal{Y} \cap \left({\mathcal{N}(X)} \cup \mathcal{N}(\optimalset{})\right)^c$ such that $\lim_{i\rightarrow\infty} \nabla V(x_i) = 0$. Note that $\left({\mathcal{N}(X)} \cup \mathcal{N}(\optimalset{})\right)^c$ is bound. WLOG, we assume $x_i$ is convergent, since otherwise we can always take a convergent subsequence. Denote $x = \lim_{i\rightarrow\infty} x_i$. We know that $x \in \left({\mathcal{N}(X)} \cup \mathcal{N}(\optimalset{})\right)^c$ since this set is closed. Moreover, since $\nabla V(x) = 0$, we have $x\in\optimalset{}$ or $x \in X$, which leads to a contradiction.
% \end{proof}
% \subsection{Proof of \Cref{Lemma:lower bound of nablaV}} \label{section_proof_of_error_bound}
% \begin{proof}
%         Take $U(x) = \left(V(x) - V^*\right)^{1 - \frac{1}{\alpha}}$. Using \Cref{ass_PL}, one has
%         \begin{equation*}
%             |\nabla U(x)| = (1-\frac{1}{\alpha}) \left(V(x) - V^*\right)^{-\frac{1}{\alpha}} |\nabla V(x)| \geq (1-\frac{1}{\alpha}) \left(V(x) - V^*\right)^{-\frac{1}{\alpha}} \PLconstant^{\frac{1}{\alpha}} \left(V(x) - V^*\right)^{\frac{1}{\alpha}} = (1-\frac{1}{\alpha}) \PLconstant^{\frac{1}{\alpha}}.
%         \end{equation*}
%         For any point $x \notin \optimalset{}$,
%         using \citep[Lemma 2.5]{drusvyatskiy2015curves} with $K = U(x)\frac{1}{(1-\frac{1}{\alpha}) \PLconstant^{\frac{1}{\alpha}}}$, $\alpha=0$, $r = (1-\frac{1}{\alpha}) \PLconstant^{\frac{1}{\alpha}}$, one has
%         \begin{align*}
%             \dist(x, \optimalset{}) \leq&\ \frac{1}{(1-\frac{1}{\alpha}) \PLconstant^{\frac{1}{\alpha}}} U(x) = \frac{\alpha}{(\alpha-1)\PLconstant^{\frac{1}{\alpha}}}\left(V(x) - V^*\right)^{1 - \frac{1}{\alpha}} \\
%             \leq&\ \frac{\alpha}{(\alpha-1)\PLconstant^{\frac{1}{\alpha}}}\PLconstant^{\frac{1}{\alpha}-1}|\nabla V(x)|^{\alpha-1} 
%             = \frac{\alpha\PLconstant^{-1}}{(\alpha-1)}|\nabla V(x)|^{\alpha-1}.
%         \end{align*}
%         We have that \Cref{eqn:error bound} holds for some constant $\errorboundconstant$.
%     \end{proof}

\subsection{Proof of \Cref{lemma_Lyapunov_alpha_PL}} \label{proof_of_Lyapunov_alpha_PL}
\paragraph{Subdomain $\Xi_{1} = \left\{x: 2R_0 \leq \|x\| \right\}$.}
% To establish the inequality (\ref{eqn:Lyapunov}), first consider $x$ such that $|x| \geq 2 R_0$, where 
Recall $R_0$ from \Cref{ass_growth_V}. One has 
    \begin{equation*}
        \dist(x, \optimalset{}) \geq |x| - R_0\ \Rightarrow\ \dist(x, \optimalset{}) \geq \frac{1}{2}|x|.
    \end{equation*}
    From \Cref{ass_growth_V}, one has 
    \begin{equation*}
        \frac{\generator W}{\epsilon W} \leq \frac{C_g|x|^{2}}{2\epsilon} - \frac{1}{4\epsilon^2}\errorboundconstant^2 \cdot \dist^{2}(x, \optimalset{}) \leq \frac{C_g|x|^{2}}{2\epsilon} -  \frac{1}{64\epsilon^2}\errorboundconstant^2|x|^2.
    \end{equation*}
    Recall that $\epsilon \leq \frac{\errorboundconstant^2}{64C_g}$ and $|x| \geq 2R_0 \geq R_0$. We have
    \begin{equation*}
    \frac{\generator W}{\epsilon W} \leq \frac{C_g|x|^{2}}{2\epsilon} -  \frac{1}{64\epsilon^2}\errorboundconstant^2|x|^{2} \leq - \frac{\errorboundconstant^2|x|^{2}}{128} \frac{1}{\epsilon^2} \leq - \frac{\errorboundconstant^2R_0^{2}}{128} \frac{1}{\epsilon^2}.
    \end{equation*}
    Now we can select the parameter $\sigma$ in \eqref{eqn:Lyapunov} as 
    ${\sigma_{1} = \frac{\errorboundconstant^2R_0^{2}}{128} \frac{1}{\epsilon^2}}$ in this region. Consequently we can establish the inequality (\ref{eqn:Lyapunov}) for $|x| \geq 2R_0$.

\paragraph{Subdomain $\Xi_{2} = \left\{x: \sqrt{C\epsilon} \leq \dist(x, \optimalset{}) \text{ and } \|x\|\leq 2R_0 \right\}$.}
    % We then consider $|x| \leq 2R_0$. 
    % Since $V \in \mathcal{C}^2$, we have that $M_\Delta := \sup_{|x|\leq 2R_0} |\Delta V(x)| < \infty$.
    % To bound the gradient norm $\|\nabla V(x)\|$, 
    There are three cases.
    \begin{itemize}
        \item  
        $x$ is in the $R_1$ neighborhood of $X$, where we recall that $X$ is the collection of all local maxima of $V$ in \Cref{lemma_property_of_V}: Using the second property of \Cref{lemma_property_of_V}, we have
        \begin{equation*}
            \frac{\generator W}{\epsilon W} \leq -\frac{d\mu^{-}}{2\epsilon}.
        \end{equation*}
        Note that the value \red{$\sigma^1_2 = \frac{d\mu^{-}}{2\epsilon}$} in this region.
        \item $x \in \mathcal{N}(\optimalset{})$: Recall $M_\Delta$ from \Cref{lemma_useful_constants}.
        Using the error bound in \Cref{lemma_useful_constants}, one has
        \begin{equation*}
            \frac{\generator W}{\epsilon W} \leq \frac{M_\Delta}{2\epsilon} - \frac{1}{4\epsilon^2}\PLconstant^2\cdot\dist^{2}(x, \optimalset{}).
        \end{equation*}
        Consequently, we can establish the inequality (\ref{eqn:Lyapunov}) for $x\in\mathcal{N}(\optimalset{})$ with $C = \frac{4M_\Delta}{\PLconstant^2}$ by
        \begin{equation*}
        - \frac{1}{4\epsilon^2}\PLconstant^2\cdot\dist^{\frac{2}{\alpha-1}}(x, \optimalset{}) \leq -\frac{M_\Delta}{\epsilon}.
        \end{equation*}
        Note that the value \red{$\sigma^2_2 = \frac{M_\Delta}{2\epsilon}$} in this region.
        \item $x$ is in the compact set, but not in the above two cases, i.e. 
        \begin{equation*}
            x \in \left(\{x:\dist(x, X) \leq R_1 \}\cup \mathcal{N}(\optimalset{})\right)^c \cap \{x:\|x\| \leq 2R_0\}.
        \end{equation*}
        According to \Cref{lemma_property_of_V}, there exists a constant lower bound of $\|\nabla V(x)\| \geq g_0 > 0$ in this regime.
        One has
        \begin{equation*}
            \frac{\generator W}{\epsilon W} \leq \frac{M_\Delta}{2\epsilon} - \frac{g_0^2}{4\epsilon^2} \leq -\frac{M_\Delta}{4\epsilon}.
        \end{equation*}
        For a sufficiently small $\epsilon$, such that  $\frac{g_0^2}{4\epsilon^2} \geq \frac{dL}{\epsilon}$.
        Note that the value \red{$\sigma^3_2 = \frac{M_\Delta}{4\epsilon}$} in this region.
    \end{itemize}
    WLOG, we assume $\frac{M_\Delta}{4\epsilon} \geq \frac{d\mu^{-}}{2\epsilon}$ (otherwise simply set $\mu^{-} = M_\Delta/2d$ in \Cref{lemma_property_of_V}). We have $\sigma_2 = \min\{\sigma^1_2, \sigma^2_2, \sigma^3_2\} = \frac{d\mu^{-}}{2\epsilon}$.

    \paragraph{Global estimation of $\sigma$.}
    Since we need \cref{eqn_requirement_Lyapunov} to hold on $\R^d \backslash U$, we take
    \begin{equation}\label{Parameter: sigma}
    \sigma = \inf_{i\in\{1, 2\}} \sigma_i
    % = \frac{\errorboundconstant^2}{4\epsilon^2} \bigg( C\epsilon\bigg)^{\frac{2}{\alpha}}.
    \end{equation}
    % Clearly, this estimation also holds for $U$ since $\R^d \backslash U \subset \R^d \backslash \optimalset{(C\epsilon)^{\frac{\alpha-1}{\alpha}}}$.

    \paragraph{Estimation of $b$.}
    Since we will pick $U \subseteq \optimalset{\sqrt{C\epsilon}}$, from \Cref{lemma_property_of_V}, we can set
    \begin{equation}\label{Parameter:b}
        b:= \sigma + \frac{M_\Delta}{2\epsilon}.
    \end{equation}

\section{Proof of results in \Cref{section_stability_neumann_eigenvalue}}
\subsection{Proof of \Cref{lemma_wely's formula}}\label{proof_of_lemma_wely's formula}
\begin{proof}
Based on the representation of local coordinate \eqref{local coordinate decomposition}, we use the formula of change of variables, 
\begin{equation*}
\int_{T(\epsilon)} \phi(y) dy = \int_{T(\epsilon)} \phi(y) |J(u,r)|du dr,
\end{equation*}
where $J(u,r)$ is Jacobian determinant of change of variables of differmorphism $y \rightarrow m + \nu$ defined by \eqref{local coordinate decomposition}. We can easily compute the Jacobian determinant $J(u,t)$ on each local chart $(\Gamma, \phi)$ as 
\begin{equation*}
\begin{aligned}
J(u,r) = & \Big| \Big[ \frac{\partial y}{\partial u^1},...,\frac{\partial y}{\partial u^k}, \frac{\partial y}{\partial r^{k+1}},...,\frac{\partial y}{\partial r^d} \Big] \Big| \\ = & \Big| \Big[ \frac{\partial y}{\partial u^1},...,\frac{\partial y}{\partial u^k}, \N_{k+1},...,\N_d \Big] \Big| \\ = & \Big| \Big[ \Big(\frac{\partial \M}{\partial u^1} + \sum_{l=k+1}^d r^l \frac{\partial \N_l}{\partial u^1}\Big),...,\Big(\frac{\partial \M}{\partial u^k} + \sum_{l=k+1}^d r^l \frac{\partial \N_l}{\partial u^k} \Big), \N_{k+1},...,\N_d \Big] \Big| \\ = & \Big| \Big[ \Big( \M_1 + \sum_{l = k+1}^d r^l \N_{l,1} \Big),...,\Big(\M_k + \sum_{l = k+1}^{d} r^l \N_{l,k} \Big), \N_{k+1},...,\N_d \Big] \Big|. 
\end{aligned}
\end{equation*}
We emphisize that these vector fields
\begin{equation*}
\{ \M_i \}_{i=1}^k, \ \ \ \ \{\N_i\}_{i=1}^{d-k}, \ \ \ \ \{ \{\N_{l,j}\}_{l=k+1}^{d} \}_{j=1}^k,
\end{equation*}
only depend on variable $u \in \Gamma \subset \R^k$, i.e. they are only defined on Riemannian submanifold $S$. Moreover, each vector at the point $m = \M(u) \in \optimalset{}$ is a linear combination of these $d$ basic vectors $\{ \{ \M_i(u) \}_{i=1}^k, \{\N_l(u)\}_{l=k+1}^d \}$. Then we have    
\begin{equation*}
\N_{l,j} = \sum_{i=1}^k T_j^i(l) \M_i + \sum_{s=k+1}^{d} T_j^s(l) \N_s.
\end{equation*}
By constrain \eqref{constrain of N-1} and \eqref{constrain of N-2}, we can easily compute that 
\begin{equation*}
\left\{
\begin{aligned}
& T_j^i(l) = \sum_{p = 1}^k g^{ip}G_{pj}(l) = G_j^i(l) & i = 1,...,k, \\
& T_j^s(l) \equiv 0,  & s = k+1,...,d, 
\end{aligned}
\right.
\end{equation*}
where $(g^{ip}) = (g_{ip})^{-1}$ is the inverse matrix of $k \times k$ matrix $(g_{ij})$ in the Riemannian metric tensor \eqref{the first fdmt form}. Hence we have
\begin{equation*}
\N_{l,j} = \sum_{i=1}^k G_j^i(l) \M_i.
\end{equation*}
Now we back to the original computation,
\begin{equation*}
\begin{aligned}
\int_{T(\epsilon)} \phi(y) dy = & \int_{T(\epsilon)} \phi(y) |J(u,r)| dudr \\ = & \int_{T(\epsilon)} \phi(y)\Big|\text{det} \Big(\big(I_k + \sum_{l = k+1}^d r^l \tilde{G}(l) \big)\Big[\M_1,...,\M_k,\N_{k+1},...,\N_d \Big] \Big) \Big|du dr,
\end{aligned}
\end{equation*}
we observe that
\begin{equation*}
\text{det} \Big( \Big[\M_1,...,\M_k,\N_{k+1},...,\N_d \Big]^{T} \Big[\M_1,...,\M_k,\N_{k+1},...,\N_d \Big] \Big) = \det(g)
\end{equation*}
by \eqref{constrain of N-1}, \eqref{constrain of N-2} and \eqref{the first fdmt form}, and the determinant $\text{det}(g)$ does not depend on variable $r \in B(\epsilon)\subset \R^{d-k}$, we finally have 
\begin{equation*}
\begin{aligned}
\int_{T(\epsilon)} \phi(y) dy =  & \int_{T(\epsilon)} \phi(y) \Big|\text{det} \Big( \big(I_k + \sum_{l = k+1}^d r^l \tilde{G}(l) \big)\Big[\M_1,...,\M_k,\N_{k+1},...,\N_d \Big] \Big) \Big| dudr \\ = & \int_{T(\epsilon)} \phi(y)\Big |\text{det} \big(I_k + \sum_{l = k+1}^d r^l \tilde{G}(l) \big) \Big| \sqrt{\det(g)}
dudr \\ = & \int_{\Gamma} \Big\{ \int_{B(\epsilon)} \phi(y) \Big|\text{det} \big(I_k + \sum_{l = k+1}^d r^l \tilde{G}(l) \big)\Big| dr^{k+1}...dr^d \Big\} d\M(u),
\end{aligned}
\end{equation*}
now we finish the proof.
\end{proof}

\subsection{Proof of \Cref{lemma_transformation}}\label{proof_of_lemma_transformation}
\begin{proof}
For the direction of the parameter $u^i$, by chain rule we have
\begin{equation*}
\begin{aligned}
\nabla_{u^i} \phi(y(u,r)) = \nabla_y \phi \cdot \frac{\partial y}{\partial u^i} = & \nabla_y \phi \cdot [\frac{\partial \M}{\partial u^i} + \sum_{l = k+1}^d r^l \frac{\partial \N_l}{\partial u^i}] \\ = & \nabla_y \phi \cdot [\M_i + \sum_{l = k+1}^d r^l \frac{\partial \N_l}{\partial u^i}] \\ = & \nabla_y \phi \cdot [\M_i + \sum_{l=k+1}^d r_l \sum_{j=1}^k G_i^j(l) \M_j] \\ = & \nabla_y \phi \cdot [\M_1,...,\M_d] \cdot [ I_k + \sum_{l=k}^d r^l \tilde{G}(l)]. 
\end{aligned}
\end{equation*}
For the direction of the parameter $t^i$, by chain rule we have
\begin{equation*}
\begin{aligned}
\nabla_{r^i} \phi(y(u,r)) = \nabla_y \phi \cdot \frac{\partial y}{\partial r^i} = \nabla_y \phi \cdot \N_i.
\end{aligned}
\end{equation*}
Combining these two results, we have
\begin{equation*}
\nabla_{(u,r)} \phi(y(u,r)) = \nabla_y \phi \cdot [\M_1,...,\M_k, \N_{k+1},...,\N_d] \cdot
\left[ 
\begin{matrix} 
I_k + \sum_{l=k+1}^d r
_l \tilde{G}(l) & 0 \\
0 & I_{d-k} 
\end{matrix}
\right].
\end{equation*}
Now we finish the proof.
\end{proof}

\subsection{Proof of \Cref{proposition_stability}}\label{proof_of_proposition_stability}
\begin{proof}
Our main idea to describe the asymptotic behavior of \Poincare\ constant $\lambda_1(T(\epsilon))$ is try to obtain the lower and upper bound of $\lambda_1(T(\epsilon))$ based on the first eigenvalue of Laplacian-Beltrami operator on a trivial product Riemannian manifold $(\optimalset{}\times B(\epsilon), g_{\optimalset{}} + g_{B(\epsilon)})$ and small $\epsilon$ perturbation issues. Here $g_{B(\epsilon)} = i_{B(\epsilon)}^{\ast}(g_E)$ is the standard Riemannian metric on $B(\epsilon) \subset \R^{d-k}$, i.e. the pullback of $g_E$ by the including map $i_{B(\epsilon)}$. In the next
we try to show 
\begin{equation*}
\lambda_1(T(\epsilon)) \sim \lambda_1(S \times B(\epsilon)) + O(\epsilon).
\end{equation*}

We start from the min-max formula and divide the estimates into two parts.

\begin{itemize}

	\item The lower and upper bound of $L^2$ norm in $(T(\epsilon), g_{E})$ based on the $L^2$ norm in $(\optimalset{} \times B(\epsilon), g_{\optimalset{}} + g_{B(\epsilon)})$ and $\epsilon$ perturbation. By expansion formula of determinant,
	\begin{equation*}
	\text{det} \big(I_k + \sum_{l = k+1}^d r^l \tilde{G}(l) \big) = 1 + \sum_{i = 1}^k \sum_{l = k+1}^d r^l G_i^i(l) + \dots,
	\end{equation*} 
	and combine with bounded condition in \Cref{ass_optimalset_interior} about the second fundamental form $\Pi$ for Riemannian submanifold $\optimalset{}$, we have
	\begin{equation*}
	1 - A_2 \epsilon \leq \text{det} \big(I_k + \sum_{l = k+1}^d r^l \tilde{G}(l) \big) \leq 1 + A_2 \epsilon.
	\end{equation*}
	for some constant $A_2 = A_2(d, k, \tilde{G}(l))$. Using integral formula in \Cref{lemma_wely's formula}, we have
	\begin{equation*}
	\begin{aligned}
	& \int_{T(\epsilon)}|\phi(y)|^2 dy \leq (1 + A_2 \epsilon) \int_{\Gamma} \Big\{ \int_{B(\epsilon)} |\phi(y)|^2 dr^{k+1}...dr^{d} \Big\} d\M, \\ 
	& \int_{T(\epsilon)}|\phi(y)|^2 dy \geq (1 - A_2 \epsilon) \int_{\Gamma} \Big\{ \int_{B(\epsilon)} |\phi(y)|^2 dr^{k+1}...dr^{d} \Big\} d\M. \\
	\end{aligned}
	\end{equation*} 
	
	\item The lower and upper bound of Dirichlet energy in $(T(\epsilon), g_E)$ based on the Dirichlet energy in $(\optimalset{} \times B(\epsilon), g_{\optimalset{}} + g_{B(\epsilon)})$ and $\epsilon$ perturbation. By \Cref{lemma_transformation}, we have
	
	\begin{equation*}
		\nabla_y \phi(y) = \nabla_{(u,r)} \phi(y(u,r)) \cdot
		\left[ 
		\begin{matrix} 
		(I_k + \sum_{l=k+1}^d r^l \tilde{G}(l))^{-1} & 0 \\
		0 & I_{d-k} 
		\end{matrix}
		\right]
		\cdot [\M_1,...,\M_k, \N_1,...,\N_{d-k}]^{-1},
	\end{equation*} 

	then we obtain 
    \begin{equation*}
    |\nabla_y \phi(y)|^2 = \nabla_{(u,r)} \phi(y(u,r)) \cdot
    \left[ 
    \begin{matrix} 
	(I_k + \sum_{l=k+1}^d r^l \tilde{G}(l))^{-1}(g)^{-1} (I_k + \sum_{l=k+1}^d r^l \tilde{G}(l))^{-1}& 0 \\
    0 & I_{d-k} 
    \end{matrix}
    \right]
    \cdot \nabla_{(u,r)} \phi(y(u,r))^{T}. 
    \end{equation*}

	Selecting $t$ small enough and using bounded condition in \Cref{ass_optimalset_interior}, we have
    \begin{equation*}
    \left\{
    \begin{aligned}
    & |\nabla_y \phi(y)|^2 \leq (1 + A_1 \epsilon)(\nabla_u \phi(y(u,r)) (g)^{-1} \nabla_u \phi(y(u,r))^T + |\nabla_r \phi(y(u,r))|^2), \\
    & |\nabla_y \phi(y)|^2 \geq (1 - 
    A_1 \epsilon)(\nabla_u \phi(y(u,r)) (g)^{-1} \nabla_u \phi(y(u,r))^T + |\nabla_r \phi(y(u,r))|^2),
    \end{aligned}
    \right.
    \end{equation*}
	for some constant $A_1 = A_1(d, k, \tilde{G}(l)) > 0$. Recall the duality between $TS$ and $T^{\ast}S$ induced by metric $g$,
	\begin{equation*}
	d^u\phi = \sum_{i,j=1}^k g^{ij} \frac{\partial \phi}{\partial u^j} \frac{\partial}{\partial u^i} = \sum_{i=1}^k g^{ij} \nabla_{u^j} \phi,
 	\end{equation*}
	where $d^u$ is external derivative associated with parameter $u$, then we have 
	\begin{equation*}
	\begin{aligned}
		\int_{T(\epsilon)}|\nabla_y \phi(y)|^2 dy = & \int_{\optimalset{}} \Big\{ \int_{B(\epsilon)} |\nabla_y \phi(y)|^2 \Big|\text{det} \big(I_k + \sum_{l = k+1}^d r^l \tilde{G}(l) \big) \Big|dr^{k+1}...dr^d \Big\} d\M \\ \leq & (1 + A_2 \epsilon) \int_{\optimalset{}} \Big\{ \int_{B(\epsilon)} |\nabla_y \phi(y)|^2 dr^{k+1}...dr^d \Big\} d\M \\ \leq & (1 + A_2 \epsilon)(1 + A_1 \epsilon) \int_{\optimalset{}} \Big\{ \int_{B(\epsilon)} (|d^u \phi|^2 + |\nabla_r \phi(y)|^2) dr^{k+1}...dr^d \Big\} d\M.
	\end{aligned}
	\end{equation*}
    Similarly, we also have
    \begin{equation*}
         \int_{T(\epsilon)}|\nabla_y \phi(y)|^2 dy \geq (1 - A_2 \epsilon)(1 - A_1 \epsilon) \int_{\optimalset{}} \Big\{ \int_{B(\epsilon)} (|d^u \phi(y)|^2 + |\nabla_r \phi(y)|^2) dr^{k+1}...dr^d \Big\} d\M.
    \end{equation*}

\end{itemize}

Combining all these estimates together, we have
\begin{equation*}
\begin{aligned}
\frac{\int_{T(\epsilon)}|\nabla_y \phi(y)|^2 dy}{\int_{T(\epsilon)} |\phi(y)|^2 dy } \leq (1 + A_1 \epsilon)(1 + A_2 \epsilon)^2 \frac{\int_{\optimalset{}} \Big\{ \int_{B(\epsilon)} (|d^u\phi(y)|^2 + |\nabla_r \phi(y)|^2) dr^{k+1}...dr^d \Big\} d\M}{\int_{\optimalset{}} \Big\{ \int_{B(\epsilon)} |\phi(y)|^2 dr^{k+1}...dr^{d} \Big\} d\M}
\end{aligned}
\end{equation*}
and
\begin{equation*}
\begin{aligned}
\frac{\int_{T(\epsilon)}|\nabla_y \phi(y)|^2 dy}{\int_{T(\epsilon)} |\phi(y)|^2 dy } \geq (1 - A_1 \epsilon)(1 - A_2 \epsilon)^2 \frac{\int_{\optimalset{}} \Big\{ \int_{B(\epsilon)} (|d^u \phi(y)|^2 + |\nabla_r \phi(y)|^2) dr^{k+1}...dr^d \Big\} d\M}{\int_{\optimalset{}} \Big\{ \int_{B(\epsilon)} |\phi(y)|^2 dr^{k+1}...dr^{d} \Big\} d\M}.
\end{aligned}
\end{equation*}
These two kinds of estimates imply us that the eigenvalues of Laplacian-Beltrami operator on $T(\epsilon)$ is equivalent to the eigenvalues of Laplacian-Beltrami operator of the following product Riemannian manifold
\begin{equation*}
(\optimalset{} \times B(\epsilon), g_{\optimalset{}} + g_{B(\epsilon)}).
\end{equation*}
It is easy to know that 
\begin{equation*}
\lambda_1(S \times B(\epsilon)) = \min \big\{ \lambda_1(\optimalset{}, g_{\optimalset{}}), \lambda_1(B{(\epsilon)}, g_{B(\epsilon)}) \big\}.
\end{equation*}
Now we complete our discussion of this problem and get conclusion
\begin{equation*}
\lambda_1(\optimalset{})(1 - B\epsilon) \leq \lambda(T(\epsilon)) \leq \lambda_1(\optimalset{})(1 + B \epsilon),
\end{equation*}
for some constant $B = B(A_1, A_2) > 0$ when $\epsilon$ small enough.
\end{proof}

\section{Proof of results in \Cref{section of main result}}

\subsection{Proof of \Cref{thm_main}} \label{proof_thm_main}
The proof need to combine  
Lyapunov approach in \Cref{section_reduction_to_neumann_eigenvalue} with spectral stability analysis in \Cref{section_stability_neumann_eigenvalue}. Recall the final result in \Cref{corollary_reduce_PI_to_Neumann_eigenvalue}, 
    \begin{equation*}
    \PIconstant_\GibbsMeasure \geq \frac{1}{2}\exp(-\bar C) \NeummanEigenvalue(U)
    \end{equation*}
with $U = \optimalset{\sqrt{C\epsilon}}$.
Let us focus on dealing with $\NeummanEigenvalue(U)$ by conclusions in \Cref{section_stability_neumann_eigenvalue},

% \begin{enumerate}

    % \item 
    % Case $(\Circle)$: $\optimalset{}$ is a submanifold of $\R^d$. 
    We use \Cref{proposition_stability} with 
    $\tilde \epsilon = \sqrt{C \epsilon}$, then we have
    \begin{equation*}
    \Eigenvalue(\optimalset{})(1 - B c(C \epsilon)^{\frac{1}{2}}) \leq \NeummanEigenvalue(U) \leq \Eigenvalue(\optimalset{})(1 + B c(C \epsilon)^{\frac{1}{2}}).
    \end{equation*}
    We finally have
    \begin{equation*}
    \PIconstant_\GibbsMeasure \geq \frac{1}{2}\exp(-\bar C) \NeummanEigenvalue(U) \geq \frac{1}{4}\exp(-\bar C) \Eigenvalue(\optimalset{}),  
    \end{equation*}
    when $\epsilon \leq \frac{1}{C} \Big( \frac{1}{2B}\Big)^2$ is small enough.
    Now we finish the proof.

\end{document}